\documentclass{siamart190516}
\usepackage[utf8]{inputenc}
\usepackage[T1]{fontenc}

\usepackage{algpseudocode}
\usepackage{caption}
\usepackage{enumitem}
\usepackage{float}
\usepackage{ifthen}
\usepackage{math}
\usepackage[numbers, sort&compress]{natbib}
\usepackage{notation}
\usepackage{subcaption}
\usepackage{tikzconfig}
\usepackage{url}
\usepackage[normalem]{ulem}

\usepackage[a-1b]{pdfx}

\newcommand{\blue}[0]{}

\newboolean{IsArxiv}
\setboolean{IsArxiv}{true}

\title{An Algorithm for Local Transverse Feedback Linearization}
\author{%
  Rollen S. D'Souza%
  \thanks{%
    Department of Electrical and Computer Engineering, University of Waterloo, Waterloo, ON, Canada.
    (\email{rollen.dsouza@uwaterloo.ca}, \email{cnielsen@uwaterloo.ca})%
    \funding{Rollen S. D'Souza is supported by the Ontario Graduate Scholarship (O.G.S.).}
  }
  \and
  Christopher Nielsen\footnotemark[1]
}
\date{\today}

\begin{document}
\bibliographystyle{siamplain}

\maketitle
\begin{abstract}
  Given a multi-input, nonlinear, time-invariant, control-affine system and a closed, embedded submanifold \texorpdfstring{\(\Manifold{N}\)}{N}, the local transverse feedback linearization (TFL) problem seeks a coordinate and feedback transformation such that, in transformed coordinates, the dynamics governing the system's transverse evolution with respect to \texorpdfstring{\(\Manifold{N}\)}{N} are linear, time-invariant and controllable.
  The transformed system is said to be in the TFL normal form.
  Checkable necessary and sufficient conditions for this problem to be solvable are known, but, unfortunately, the literature does not present a prescription that constructs the required transformation for multi-input systems.
  In this article we present an algorithm that produces a virtual output of suitable vector relative degree that, using input-output feedback linearization, puts the system into TFL normal form.
  The procedure is based on dual conditions for TFL and is fundamentally different from existing methods, such as the GS and Blended algorithms, because of the ``desired'' zero dynamics manifold \texorpdfstring{\(\Manifold{N}\)}{N}.
  The proposed algorithm is the first to take into consideration the desired zero dynamics.
\end{abstract}

\section{Introduction}
\label{sec:intro}
Nonlinear control theory has a long geometric tradition.
Hermann first injected geometric ideas into control theory when studying the accessibility problem~\cite{Hermann1963}.
Within decades, a flurry of articles from a number of notable academics used the language of differential geometry to aid in the design of feedback linearizing controllers, ascertain controllability, and find controlled-invariant sets and distributions~\cite{Hermann1981, Hermann1988, Hermann1989, Hirschorn1981, Brockett1983}.
These results have since been synthesized into mainstream nonlinear control texts like~\cite{Sastry1999, Isidori1995, Nijmeijer2016}.

The exact full-state feedback linearization problem is one issue tackled and resolved by taking a geometric approach.
The problem asks whether there exists a change of coordinates \(\xi = \Phi(x)\) and feedback transformation \(v = \alpha(x) + \beta(x)\, u\) that locally transforms the nonlinear control-affine system
\[
  \dot{x}(t) = f(x(t)) + g(x(t))\, u(t),
\]
into a linear, controllable system
\[
  \dot{\xi}(t) = A\, \xi(t) + B\, v(t).
\]
Hunt, Su and Meyer presented in~\cite{HuntSuMeyer1983} necessary and sufficient conditions upon which such a transformation exists.
If the conditions are satisfied, then the transformation \(\Phi\) can be found by solving a large system of partial differential equations (PDE).

\begingroup\blue
The difficulty solving this system of PDEs was clear, and within a few years attempts were made to alleviate the difficulty.
The first attempt was made in~\cite{Gardner1991} by Gardner and
Shadwick in the case where the controllability indices of \((A, B)\) in the target system are distinct.
They completed their work in~\cite{Gardner1992} for the general case.
The GS algorithm, their proposed approach, minimizes the number of integrations required to find the transformation \(\Phi.\)
\endgroup
This remarkable work soon inspired attempts at algorithmic solutions to other nonlinear control design problems such as dynamic feedback linearization~\cite{Tilbury1994, Aranda-Bricaire1995} and differential flatness~\cite{Schoberl2014}.

An interesting problem arose in the late 1990s with application to motion control problems:  given a controlled invariant submanifold \(\Manifold{N}\) of the state-space, when is it possible to feedback linearize those dynamics that act transverse\footnote{\blue The dynamics are not tangent to the set.} to \(\Manifold{N}\)?
This problem, known as the transverse feedback linearization (TFL) problem, differs from partial feedback linearization in that it \emph{starts} with a desired zero dynamics manifold \(\Manifold{N}\) and asks to find a state and feedback transformation that locally linearizes dynamics transverse to it.
When \(\Manifold{N}\) is an orbit, necessary and sufficient conditions upon which this could be done was found by Banaszuk and Hauser in~\cite{Banaszuk1995}.
Nielsen and Maggiore then presented conditions for when \(\Manifold{N}\) is a general submanifold~\cite{Nielsen2008}.
They also suggested a semi-constructive procedure that can be used to find the aforementioned state and feedback transformation in simple cases.

Unfortunately, their procedure relies on the existence of an integrable distribution adapted to the tangent space of \(\Manifold{N}\) and to the specific system structure.
Finding this distribution can be difficult even in the single-input case as demonstrated in~\cite{DSouza2021}.
Inspired by~\cite{Gardner1992}, the authors of this work sought an algorithmic procedure that could, at a minimum, produce the integrable distribution.
The algorithm for the single-input case was presented in~\cite[Remark 2]{DSouza2021}.
This article generalizes that work and presents an algorithm that produces the state and feedback transformation in the multi-input case.

The algorithm we propose possesses a number of features of note.
First, it produces its own certificate: the transverse output for which input-output feedback linearization can be performed to locally feedback linearize the dynamics transverse to \(\Manifold{N}.\)
This is in-line with known algorithms for exact full-state feedback linearization.
Secondly, the algorithm provides a geometric intuition for the algebraic adaptation process performed on the derived flag.
We show that the adaptation process amounts to producing a sequence of descending zero dynamics manifolds converging upon the desired zero dynamics manifold \(\Manifold{N}.\)
This not only gives a geometric perspective to our algorithm, but also to the GS algorithm~\cite{Gardner1992} and ``Blended Algorithm''~\cite{Mehra2014} since, when \(\Manifold{N} = \{ x_0 \}\) is a one-point set, the TFL problem coincides with the exact full-state feedback linearization problem~\cite[Corollary 3.3]{Nielsen2008}.
Unlike previously published algorithms, the proposed algorithm begins with a desired zero dynamics chosen \emph{a-priori} and finds a virtual output (the so-called transverse output) so that its zero dynamics coincide with the desired zero dynamics.
This is something that \emph{cannot} be done with the GS and Blended algorithms.
This is also not possible with the algorithm for partial feedback linearization presented in~\cite{Marino1986} as the resulting zero dynamics manifold is not fixed before-hand;
the resulting zero dynamics is not guaranteed to equal \(\Manifold{N}.\)
This article directly addresses this with the proposed algorithm, leveraging the conditions under which a desired zero dynamics manifold may be transverse feedback linearized, that produces the required output for the desired zero dynamics manifold.

\begingroup\blue
The ability to find an output with a well-defined relative degree whose associated zero dynamics equals a given set is an important contribution to control problems whose specification demands the state be driven to the aforementioned set.
A notable example is path following for robotic systems; this specification can be cast as the problem of driving the system state to a controlled invariant set describing all the motions of the robot along the path~\cite{Gill2015, DSouza2021c}.

Another example is safety control;
safety is enforced by making an appropriate set controlled invariant and control barrier functions (CBFs) are widely used~\cite{Ames2019}.
The results of this article can be used to compute functions with a well-defined relative degree on the boundary of the safety set --- even when the CBF that defines the safety specification doesn't.
Our results, combined with recent work on CBFs with higher relative degree~\cite{Nguyen2016, Xiao2019}, may find applications in safety critical control.
\endgroup

This article is organized as follows.
After presenting our notation, we formally introduce the TFL problem in Section~\ref{sec:tfl}.
Section~\ref{sec:mathbg} then presents a preliminary lifting of the problem into a setting that includes time and control in a way similar to both~\cite{Gardner1992} and~\cite{DSouza2021}.
\begingroup\blue
Section~\ref{sec:main} then presents the conditions for multi-input TFL in the exterior differential setting and proves these conditions are necessary.
The supporting results are developed in Section~\ref{sec:tech}.
Section~\ref{sec:proposedalgorithm} then returns to prove that the conditions for TFL are sufficient using the proposed algorithm.
The proposed algorithm itself is simplified and presented in Section~\ref{sec:thealgorithm}.
We conclude our work with a non-trivial academic example that demonstrates our algorithm.
\endgroup

\subsection{Notation}
\label{sec:notation}
\begingroup\blue
The set of natural numbers is denoted by \(\Naturals\) and the set of real numbers by \(\Real.\)
If \(A\) is a finite set, then \(\Card(A)\) denotes its cardinality.
If \(\Manifold{M}\) is a smooth \((\CInfty)\) differentiable manifold of dimension \(m,\) then we denote by \(\CInfty(\Manifold{M})\) the ring of smooth real-valued functions on \(\Manifold{M}.\)
If \(p \in \Manifold{M},\) then \(\TangentSpaceAt{\Manifold{M}}{p}\) denotes the tangent space at the point \(p\) and its dual, the cotangent space, is denoted \(\CotangentSpaceAt{\Manifold{M}}{p}.\)
The tangent and cotangent bundles of \(\Manifold{M}\) are written \(\TangentSpaceOf{\Manifold{M}}\) and \(\CotangentSpaceOf{\Manifold{M}}\) respectively.
If \((\OpenSet{U};\) \(x^1,\) \(\ldots,\) \(x^m)\) is a chart of \(\Manifold{M},\) then for each \(p \in \OpenSet{U}\) the basis of vectors for \(\TangentSpaceAt{\Manifold{M}}{p}\) induced by the chart is denoted by \(\left.\derivation{x^1}\right|_p,\) \(\ldots,\) \(\left.\derivation{x^m}\right|_p.\)
The vector fields \(\derivation{x^1},\) \(\ldots,\) \(\derivation{x^m}\) form a local frame for \(\TangentSpaceOf{\Manifold{M}}.\)
The unique dual basis for \(\CotangentSpaceAt{\Manifold{M}}{p}\) induced by the chart is denoted \(\D x^1_p,\) \(\ldots,\) \(\D x^m_p.\)
If \(H: \Manifold{M} \to \Manifold{N}\) is a smooth map between manifolds, then the pushforward at \(p\) is \(\left.\J H\right|_p: \TangentSpaceAt{\Manifold{M}}{p} \to \TangentSpaceAt{\Manifold{N}}{H(p)}\) and the pullback at \(p\) is the dual map \(\left.\J H\right|_p^*: \CotangentSpaceAt{\Manifold{N}}{H(p)} \to \CotangentSpaceAt{\Manifold{M}}{p}.\)

The set of smooth sections of \(\TangentSpaceOf{\Manifold{M}}\) is denoted by \(\VectorFieldsOn{\Manifold{M}},\) whose elements define vector fields on \(\Manifold{M},\) and the set of smooth sections of \(\CotangentSpaceOf{\Manifold{M}}\) is denoted by \(\DifferentialForms{\Manifold{M}},\) whose elements define covector fields (smooth one-forms) on \(\Manifold{M}.\)
The set \(\VectorFieldsOn{\Manifold{M}}\) (resp. \(\DifferentialForms{\Manifold{M}}\)) is a real vector space, but can also be endowed with the structure of a module over the ring \(\CInfty(\Manifold{M}).\)
Let \(k \in \Naturals \cup \{0\}.\)
Define the set of smooth \(k\)-forms \(\SectionsOf{\Forms^k \CotangentSpaceOf{\Manifold{M}}}\) as sections of the bundle of \(k\)-forms and the space of all smooth forms as \(\SectionsOf{\Forms \CotangentSpaceOf{\Manifold{M}}} \defineas \bigoplus_{k=0}^{m} \SectionsOf{\Forms^k \CotangentSpaceOf{\Manifold{M}}}.\)
Let \(X \in \VectorFieldsOn{\Manifold{M}}\) and \(\omega \in \SectionsOf{\Forms^k \CotangentSpaceOf{\Manifold{M}}}.\)
The notation \(X_p \in \TangentSpaceAt{\Manifold{M}}{p}\) (\(\omega_p \in \Forms^k \CotangentSpaceAt{\Manifold{M}}{p}\)) denotes the vector \(X\) (\(k\)-form \(\omega\)) in the tangent space (forms over the cotangent space) at \(p \in \Manifold{M}.\)

If \(\Dist{G} \subseteq \TangentSpaceOf{\Manifold{M}}\) is a distribution, then its restriction to \(p \in \Manifold{M}\) is denoted \(\Dist{G}_p \subseteq \TangentSpaceAt{\Manifold{M}}{p}.\)
The set of covectors that annihilate vectors in \(\Dist{G}_p\) is \(\Ann(\Dist{G}_p)\) \(\subseteq\) \(\CotangentSpaceAt{\Manifold{M}}{p}.\)
Define \(\Ann(\Dist{G})\) \(\defineas\) \(\sqcup_{p \in \Manifold{M}} \Ann(\Dist{G}_p)\) \(\subseteq\) \(\CotangentSpaceOf{\Manifold{M}}.\)
If \(\Codist{I} \subseteq \CotangentSpaceOf{\Manifold{M}}\) is a codistribution, then its restriction to \(p\) is denoted \(\Codist{I}_p \subseteq \CotangentSpaceAt{\Manifold{M}}{p}.\)

When a distribution \(\Dist{D} \subseteq \TangentSpaceOf{\Manifold{M}}\) is smooth and regular, it is a subbundle of \(\TangentSpaceOf{\Manifold{M}}\) and, as such, can be associated to a \(\CInfty(\Manifold{M})\)-submodule of vector fields \(\Module{D}\) \(\defineas\) \(\SectionsOf{\Dist{D}}\) \(\subseteq\) \(\VectorFieldsOn{\Manifold{M}}\) comprising of smooth sections \(X \in \VectorFieldsOn{\Manifold{M}}\) that satisfy \(X_p \in \Dist{D}_p\) for all \(p \in \Manifold{M}.\)

If \(\Module{D} \subseteq \VectorFieldsOn{\Manifold{M}}\) is a \(\CInfty(\Manifold{M})\)-submodule and \(p \in \Manifold{M},\) then to this object we associate a distribution defined pointwise by,
\[
  \Dist{D}_p \defineas \{ X_p \colon X \in \Module{D} \} \subseteq \TangentSpaceAt{\Manifold{M}}{p}.
\]
We say a smooth \(k\)-form \(\omega\) annihilates \(\Module{D} \subseteq \VectorFieldsOn{\Manifold{M}}\) if it evaluates to zero when all of its arguments are vector fields from \(\Module{D}.\)
The set of smooth forms that annihilate vector fields in \(\Module{D}\) is \(\Ann(\Module{D}) \subseteq \SectionsOf{\Forms \CotangentSpaceOf{\Manifold{M}}}.\)
If \(\Ideal{I} \subseteq \SectionsOf{\Forms \CotangentSpaceOf{\Manifold{M}}}\) and \(p \in \Manifold{M},\) then to this object we associate a codistribution defined pointwise by,
\[
  \Codist{I}_p \defineas 
    \{
      \omega_p \colon
      \omega \in \Ideal{I} \cap \SectionsOf{\CotangentSpaceOf{\Manifold{M}}}
    \}
  \subseteq \CotangentSpaceAt{\Manifold{M}}{p}.
\]

Given two smooth vector fields \(X,\) \(Y \in \VectorFieldsOn{\Manifold{M}},\) their Lie bracket is \([X, Y] \in \VectorFieldsOn{\Manifold{M}}.\)
If \(\Module{D}, \Module{G} \subseteq \VectorFieldsOn{\Manifold{M}}\) are submodules, then 
\[
  [\Module{D}, \Module{G}] \defineas 
    \{
      [X, Y]
      \colon
      X \in \Module{D}, Y \in \Module{G}
    \}
  \subseteq \VectorFieldsOn{\Manifold{M}}.
\]
Similarly, if \(\Dist{D}, \Dist{G} \subseteq \TangentSpaceOf{\Manifold{M}}\) are distributions, then we can define their Lie bracket pointwise,
\[
  [\Dist{D}, \Dist{G}]_p \defineas 
    \{
      \left.[X, Y]\right|_p
      \colon
      X \in \Module{D}, Y \in \Module{G}
    \}
  \subseteq \TangentSpaceAt{\Manifold{M}}{p}.
\]
If \(\Dist{D} \subseteq \TangentSpaceOf{\Manifold{M}}\) is a distribution, then the involutive closure is \(\Inv(\Dist{D}).\)
As a matter of convenience, repeated Lie brackets are compressed using the following notation.
Let \(X,\) \(Y \in \VectorFieldsOn{\Manifold{M}}\) and define \(\Ad_X^0 Y \defineas Y\) and \(\Ad_X^1 Y \defineas [X, Y].\)
Recursively define \(\Ad_X^k Y \defineas [X, \Ad_X^{k - 1} Y]\) for all \(k > 1.\)

Let \(\ell \in \Naturals \cup \{ 0 \}.\)
If \(\omega \in \SectionsOf{\Forms^k \CotangentSpaceOf{\Manifold{M}}}\) and \(\beta \in \SectionsOf{\Forms^\ell \CotangentSpaceOf{\Manifold{M}}}\) denote their wedge product by \(\omega \wedge \beta \in \SectionsOf{\Forms^{k+\ell}\CotangentSpaceOf{\Manifold{M}}}.\)
The wedge product distributes over the addition of smooth forms, amd endows the space of smooth forms with a graded algebra structure over the ring of smooth functions.
If \(\omega^1,\) \(\ldots,\) \(\omega^\ell \in \SectionsOf{\Forms^1 \CotangentSpaceOf{\Manifold{M}}}\) are smooth one-forms, then \(\langle \omega^1,\) \(\ldots,\) \(\omega^\ell\rangle \subseteq \SectionsOf{\Forms \CotangentSpaceOf{\Manifold{M}}}\) denotes the ideal generated by \(\omega^1,\) \(\ldots,\) \(\omega^\ell\) over the aforementioned graded algebra.
The exterior derivative of \(\omega \in \SectionsOf{\Forms^{k}\CotangentSpaceOf{\Manifold{M}}}\) is \(\D \omega \in \SectionsOf{\Forms^{k+1}\CotangentSpaceOf{\Manifold{M}}}.\)
If \(\Ideal{I} \subseteq \SectionsOf{\Forms \CotangentSpaceOf{\Manifold{M}}}\) is an ideal, then the largest ideal contained in \(\Ideal{I}\) that is closed under the exterior derivative is denoted \(\Ideal{I}^{(\infty)}\);
the ideal \(\Ideal{I}^{(\infty)}\) is otherwise known as the differential closure of \(\Ideal{I}.\)
The Lie derivative of \(\omega\) along a vector field \(X \in \VectorFieldsOn{\Manifold{M}}\) is the smooth \(k\)-form \(\Lie_X \omega \in \SectionsOf{\Forms^k \CotangentSpaceOf{\Manifold{M}}}.\)
Repeated Lie derivatives of order \(j > 1\) are defined recursively by \(\Lie_X^j \omega \defineas \Lie_X^{j - 1} (\Lie_X \omega).\)

If \(H: \Manifold{M} \to \Real^\ell\) is smooth, then it can be written component-wise as \(H = (H^1,\) \(\ldots,\) \(H^\ell)\) where \(H^1,\) \(\ldots,\) \(H^\ell: \Manifold{M} \to \Real\) are smooth.
The function \(H^i: \Manifold{M} \to \Real\) may be seen as a smooth zero-form on \(\Manifold{M}\) and, as such, has an exterior derivative \(\D H^i \in \SectionsOf{\CotangentSpaceOf{\Manifold{M}}}\) and a Lie derivative, along the vector field \(X,\) \(\Lie_X H^i: \Manifold{M} \to \Real.\)
\endgroup

\section{Transverse Feedback Linearization}
\label{sec:tfl}
Consider the nonlinear, control affine system
\begin{equation}
  \label{eqn:system}
  \dot{x}(t) = \tilde{f}(x(t))
    + \textstyle\sum_{j = 1}^m \tilde{g}_j(x(t))\, u^j(t).
\end{equation}
where \(x(t) \in \Real^n,\) \(u(t) = (u^1(t),\ldots,u^m(t))\in \Real^m,\) and \(\tilde{f}, \tilde{g}_j \in \VectorFieldsOn{\Real^n}\);
component-wise, we write \(\tilde{f} = (\tilde{f}^1, \ldots, \tilde{f}^n)\) and \(\tilde{g}_j = (\tilde{g}^1_j, \ldots, \tilde{g}^n_j).\)

Fix \(\Manifold{N} \subset \Real^n\) to be a given closed, embedded submanifold of dimension \(0 \leq n^* < n\) rendered controlled invariant by a feedback \(u_*: \Manifold{N} \to \Real^m\) and fix a point \(x_0 \in \Manifold{N}.\)
The local transverse feedback linearization problem asks us to find a state and feedback transformation defined in a neighbourhood of \(x_0\) so that in the transformed coordinates the dynamics transverse to the set \(\Manifold{N}\) are linear and controllable.
More precisely, find, on an open set \(\OpenSet{U} \subseteq \Real^n\) containing \(x_0,\) a diffeomorphism \(\Phi: \OpenSet{U} \to \Phi(\OpenSet{U}),\) \((\eta, \xi) = \Phi(x),\) and feedback transformation \((v_\parallel, v_\pitchfork) \defineas \alpha(x) + \beta(x) u\) where, in the \((\eta, \xi)\)-coordinates, the nonlinear control system~\eqref{eqn:system} takes the form
\[
\begin{aligned}
  \dot{\eta}(t)
    &=
      \overline{f}(\eta(t), \xi(t))
      +
      \textstyle\sum_{j = 1}^{m - \rho_0}
        \overline{g}_{\parallel,j}(\eta(t), \xi(t))\,
        v_\parallel^j(t)
      +
      \textstyle\sum_{j = 1}^{\rho_0}
        \overline{g}_{\pitchfork,j}(\eta(t), \xi(t))\,
        v_\pitchfork^j(t),\\
  \dot{\xi}(t)
    &=
      A\, \xi(t) + \textstyle\sum_{j = 1}^{\rho_0} {b_j}\, v_\pitchfork^j(t),
\end{aligned}
\]
and \((A,\) \(\begin{bsmallmatrix} b_1 & \cdots & b_{\rho_0}\end{bsmallmatrix})\) is in Brunovsk\'y normal form.
Moreover, in \((\eta, \xi)\)-coordinates, the manifold \(\Manifold{N}\) is locally given by
\[
  \Phi(\OpenSet{U} \cap \Manifold{N})
    =
      \left\{
        (\eta, \xi) \in \Phi(\OpenSet{U})
        \colon
        \xi = 0
      \right\}.
\]
One solution to this problem is to find an output of suitable vector relative degree that vanishes on \(\Manifold{N}.\)
This was the view championed by Isidori with regards to the problem of feedback linearization.
In the spirit of this idea, the following theorem was established in~\cite{Nielsen2008}.
\begin{theorem}[{\cite[Theorem 3.1]{Nielsen2008}}]
  \label{thm:LTFL-Nielsen-3.1}
  The local transverse feedback linearization problem is solvable at \(x_0\) if, and only if, there exists an open subset \(\OpenSet{U} \subseteq \Real^n\) of \(x_0\) and smooth function \(h: \OpenSet{U} \to \Real^{\rho_0}\) so that:
  \begin{enumerate}[label={(\arabic*)}]
    \item{
      \label{thm:LTFL-Nielsen-3.1:1}
      \(
        \OpenSet{U} \cap \Manifold{N}
          \subset
            h^{-1}(0),
      \)
      and
    }
    \item{
      \label{thm:LTFL-Nielsen-3.1:2}
      the system~\eqref{eqn:system} with output \(y = h(x)\) yields a well-defined vector relative degree of \(\kappa = (\kappa_1, \ldots, \kappa_{\rho_0})\) at \(x_0\) with \(\sum_{i=1}^{\rho_0} \kappa_i = n - n^*.\)
    }
  \end{enumerate}
\end{theorem}
The theorem shows that the transverse feedback linearization problem is equivalent to the zero dynamics assignment problem with relative degree:
find an output \(h\) for system~\eqref{eqn:system} that yields a well-defined relative degree and whose zero dynamics manifold locally coincides with \(\Manifold{N}.\)
\begingroup\blue
The output \(h\) is called a (local) \emph{transverse output} with respect to \(\Manifold{N}\) at \(x_0\), or, a transverse output for short.
Theorem~\ref{thm:LTFL-Nielsen-3.1} is not particularly useful in finding the output \(h,\) but we will show it still plays a crucial role in the transverse feedback linearization algorithm presented in this article.
\endgroup

\section{Technical Preliminaries}
\label{sec:mathbg}
Define the ambient manifold as the Cartesian product
\[
  \Manifold{M} \defineas \Real \times \Real^m \times \Real^n,
\]
of time (\(\Real\)), control (\(\Real^m\)) and states (\(\Real^n\)).
Let \(\pi: \Manifold{M} \to \Real^n\) be the projection map \(\pi(t,u,x) = x\) and let \(\iota: \Real^n \to \Manifold{M}\) be the insertion map \(\iota(x) = (0,0,x).\)
These maps are used to formally define functions that are independent of the control and time variables.
A smooth function \(h: \Manifold{M} \to \Real^m\) is said to be \emph{a smooth function of the state} if the diagram
\begin{center}
  \begin{tikzcd}
    \Manifold{M}
    \arrow{rr}{\pi}
    \arrow{rrd}{h}
      &
      &
        \Real^n
        \arrow{d}{h \circ \iota}
    \\
      &
      &
        \Real^m
  \end{tikzcd}
\end{center}
commutes.
It is convenient to consider vector fields that are \(\iota\)-related to the vector fields \(\tilde{f}\) and \(\tilde{g}_j\) in~\eqref{eqn:system}.
Define
\[\blue
  {f} 
    \defineas \textstyle\sum_{i=1}^n (\tilde{f}^i\circ \pi) \Derivation{x^i}
    \in \VectorFieldsOn{\Manifold{M}},
  \quad
  {g}_j \defineas \textstyle\sum_{i=1}^n (\tilde{g}^i_j \circ \pi) \Derivation{x^i} \in \VectorFieldsOn{\Manifold{M}},
  \quad 1 \leq j \leq m,
\]
so that \(\tilde{f}(\pi(p)) = \left.\J \pi\right|_p {f}(p)\) and that \(\tilde{f}\) and \({f}\) are \(\iota\)-related vector fields.

With these constructions, the control system~\eqref{eqn:system} is differentially equivalent to the system of differential equations on \(\Manifold{M},\)
\begin{equation}
  \label{eqn:systemOnM}
  \dot{t} = 1,\quad
  \dot{x} = {f}(x) + \textstyle\sum_{j = 1}^m {g}_j(x) u^j.
\end{equation}
Furthermore, system~\eqref{eqn:system} with output \(h: \Real^n \to \Real^m\) yields a vector relative degree at \(x_0 \in \Real^n\) if, and only if, system~\eqref{eqn:systemOnM} with output \(h \circ \pi: \Manifold{M} \to \Real^m\) yields a vector relative degree at \(\iota(x_0) \in \Manifold{M}.\)
We also consider the lift~\cite{DSouza2021} of the manifold \(\Manifold{N}\) into \(\Manifold{M}\).
Define the closed, embedded submanifold
\begin{equation}
  \label{eqn:lifted-target-set}
  \Manifold{L}
    \defineas
      \{
        p = (t, u, x) \in \Manifold{M}
        \colon
        t = 0,
        x \in \Manifold{N}
      \}.
\end{equation}
Fix \(p_0\) \(\defineas\) \((0,\) \(u_*(x_0),\) \(x_0)\) \(\in\) \(\Manifold{L} \subseteq \Manifold{M}.\)
At times we will lift other submanifolds of the state-space \(\Real^n\) in the same manner as in~\eqref{eqn:lifted-target-set}.

\begingroup
\blue
We can view solutions to the control system~\eqref{eqn:systemOnM} as integral submanifolds of a distribution.
Define the smooth and regular distribution of control directions,
\begin{equation}
  \label{eqn:control-dist}
  \Dist{U}_p
    \defineas
      \SpanInline{\Real}{
        \left.\derivation{u^1}\right|_p,
        \ldots,
        \left.\derivation{u^m}\right|_p
      }
      \subseteq \TangentSpaceAt{\Manifold{M}}{p}
      ,
    \qquad
      p \in \Manifold{M},
\end{equation}
which is associated with the \(\CInfty(\Manifold{M})\)-submodule \(\Module{U} \defineas \SectionsOf{\Dist{U}}.\)
Additionally, define the smooth and regular distribution
\begin{equation}
  \label{eqn:dist-0}
  \Dist{D}^{(0)}_p
    \defineas
      \SpanInline{\Real}{
        \left.\derivation{t}\right|_p
          +
          \left.f\right|_p + \textstyle\sum_{j = 1}^m \left.{g}_j\right|_p u^j
      }
      +
      \Dist{U}_p,
\end{equation}
which is associated to the \(\CInfty(\Manifold{M})\)-submodule \(\Module{D}^{(0)} \defineas \SectionsOf{\Dist{D}^{(0)}}.\)
Observe that the vector field
\begin{equation}
  \label{eqn:vfieldsystem}
  F \defineas
    \derivation{t}
    +
    f + \textstyle\sum_{j = 1}^m {g}_j u^j
  \in \Module{D}^{(0)}
\end{equation}
is tangent to solutions of~\eqref{eqn:systemOnM}.
When \(u\) is a sufficiently regular signal, integral submanifolds of \(\Dist{D}^{(0)}\) determine solutions of~\eqref{eqn:systemOnM} and, in turn, solutions of~\eqref{eqn:system}.

Alternatively, the control system~\eqref{eqn:systemOnM} may be viewed as an exterior differential system on \(\Manifold{M}\) in the following way.
Define the smooth one-forms
\begin{equation}
  \label{eqn:omega}
  \omega^i
    \defineas
      \D x^i
      -
      (f^i(x) + \textstyle\sum_{j = 1}^m g^i_j(x)\, u^j)\,\D t
    \in \SectionsOf{\CotangentSpaceOf{\Manifold{M}}},
  \qquad
  1\leq i \leq n.
\end{equation}
The submanifolds of \(\Manifold{M}\) on which the ideal,
\begin{equation}
  \label{eqn:systemideal}
  \Ideal{I}^{(0)} \defineas \langle \omega^1, \ldots, \omega^n \rangle
  \subseteq
  \SectionsOf{\Forms\CotangentSpaceOf{\Manifold{M}}},
\end{equation}
vanishes correspond to solutions of the differential equation~\eqref{eqn:systemOnM} where \(u\) is a sufficiently regular signal --- corresponding therein to solutions of~\eqref{eqn:system}.
The ideal \(\Ideal{I}^{(0)}\) is simply, finitely, non-degenerately generated by construction since it is generated by a finite number of smooth one-forms \(\omega^i\) that are pointwise linearly independent.
It follows that the generators of \(\Ideal{I}^{(0)}\) span a smooth and regular codistribution \(\Codist{I}^{(0)} \subseteq \CotangentSpaceOf{\Manifold{M}}.\)
To this ideal \(\Ideal{I},\) we associate the object of importance in this article: the derived flag.
\begin{definition}
  \label{def:derivedflag}
  Let \(\Ideal{I}^{(0)} \subseteq \SectionsOf{\Forms\CotangentSpaceOf{\Manifold{M}}}\) be an ideal, and define the \textbf{derived ideals} by
  \begin{equation}
    \label{eqn:derive}
    \Ideal{I}^{(k + 1)}
      \defineas \{ \omega \in \Ideal{I}^{(k)} \colon \D \omega \in \Ideal{I}^{(k)} \}, \qquad k \geq 0.
  \end{equation}
  The \textbf{derived flag} of \(\Ideal{I}^{(0)}\) is the sequence of derived ideals,
  \begin{equation}
    \label{eqn:flag}
    \{ 0 \}
      \subseteq \cdots
      \subseteq \Ideal{I}^{(i+1)}
      \subseteq \Ideal{I}^{(i)}
      \subseteq \cdots
      \subseteq \Ideal{I}^{(1)}
      \subseteq \Ideal{I}^{(0)}.
  \end{equation}
  The \textbf{length} of the derived flag is the smallest \(N \in \Naturals\) such that \(\Ideal{I}^{(N)} = \Ideal{I}^{(N+1)}.\)
\end{definition}

The decreasing sequence~\eqref{eqn:flag} terminates at the differential ideal \(\Ideal{I}^{(N)} \subseteq \Ideal{I}^{(0)},\) i.e., there exists a smallest \(N \in \Naturals\) so that \(\Ideal{I}^{(N)} = \Ideal{I}^{(N+1)}.\)
The differential ideal \(\Ideal{I}^{(N)}\) is not, in general, the largest differential ideal contained within \(\Ideal{I}^{(0)}.\)
For this reason, we define the largest differential ideal contained within an ideal.
\begin{definition}
  \label{def:diffclosure}
  Let \(\Ideal{I} \subseteq \SectionsOf{\Forms\CotangentSpaceOf{\Manifold{M}}}\) be an ideal.
  The largest, differential ideal contained in \(\Ideal{I}\) is denoted \(\Ideal{I}^{(\infty)}.\)
  The ideal \(\Ideal{I}^{(\infty)}\) is said to be the \textbf{differential closure} of the ideal \(\Ideal{I}.\)
\end{definition}
The existence of the differential closure is ensured by an argument leveraging Zorn's Lemma (see~\cite[Lemma 3.9.4]{Lewis2005}).
With an abuse of notation, we denote the differential closure of the ideal \(\Ideal{I}^{(0)}\) by \(\Ideal{I}^{(\infty)}.\)
If all the ideals in the derived flag~\eqref{eqn:flag} are simply, finitely, non-degenerately generated --- as was the case for \(\Ideal{I}^{(0)}\) --- then \(\Ideal{I}^{(N)} = \Ideal{I}^{(\infty)}.\)
Consequently, we make the following convenient assumption about this flag.
\begin{assumption}
  \label{assume:regular}
  The ideals \(\Ideal{I}^{(k)}\) and the augmented ideals \(\langle \Ideal{I}^{(k)}, \D t\rangle^{(\infty)}\) are locally, simply, finitely, non-degenerately generated, for all \(k \geq 0.\)
\end{assumption}
Assumption~\ref{assume:regular} allows us to take the generators of \(\Ideal{I}^{(k)}\) and use them as a basis for a smooth and regular codistribution \(\Codist{I}^{(k)} \subseteq \CotangentSpaceOf{\Manifold{M}}.\)
In particular, \(\Codist{I}^{(k)}\) is a smooth and regular codistribution for all \(k \geq 0.\)

Of great import to this article is the notion of an output \(h: \Manifold{M} \to \Real^m \) for system~\eqref{eqn:systemOnM} yielding vector relative degree \((\kappa_1,\) \(\ldots,\) \(\kappa_m)\) at a point \(x_0.\)
Conditions for this are well-established, and are explicitly discussed on~\cite[pg. 220]{Isidori1995}.
Here we state a dual variant of these conditions that apply in the uniform vector relative degree case --- when \(\kappa_1 = \cdots = \kappa_m.\)
The proof of this result can be found in~\cite[Lemma 2.2.14]{DSouza2022}.
For brevity, it is not included in this article.
\begin{proposition}
  \label{prop:relativedegree}
  Suppose \(h^1,\) \(\ldots,\) \(h^\ell \in \CInfty(\Manifold{M}),\) \(\ell \leq m,\) are smooth functions of the state with linearly independent differentials at \(p_0.\)
  The system~\eqref{eqn:systemOnM} with output \((h^1,\) \(\ldots,\) \(h^\ell)\) yields a vector relative degree \((\kappa_1,\) \(\ldots,\) \(\kappa_1)\) at \(p_0\) if, and only if, there exists an open set \(\OpenSet{U} \subseteq \Manifold{M}\) containing \(p_0\) where,
  \begin{equation}
    \label{prop:relativedegree:a}
    \GenerateInline{\D h^1, \ldots, \D h^\ell}
    \subseteq
    \GenerateInline{\Ideal{I}^{(\kappa_1 - 1)}, \D t}^{(\infty)}
  \end{equation}
  and, at \(p_0,\)
  \begin{equation}
    \label{prop:relativedegree:b}
    \SpanInline{\Real}{
      \D h^1_{p_0}, \ldots, \D h^\ell_{p_0}
    }
    \cap
    \SpanInline{\Real}{
      \Codist{I}^{(\kappa_1)}_{p_0},
      \D t_{p_0}
    }
    =
    \{
      0
    \}.
  \end{equation}
\end{proposition}
Proposition~\ref{prop:relativedegree} only addresses those outputs with uniform relative degree because the non-uniform case involves ensuring the scalar outputs and their Lie derivatives form an adapted basis that generates the ideals in the derived flag~\eqref{eqn:flag}.
This result is useful for finding and verifying outputs that yield uniform vector relative degree.
\endgroup

\section{Main Result}
\label{sec:main}
The goal of this article is to demonstrate a constructive algorithm for producing a transverse output with respect to \(\Manifold{N}\) at \(x_0 \in \Manifold{N}.\)
The computable necessary and sufficient conditions under which this algorithm succeeds are a slight variation of those proposed in~\cite{DSouza2021}.
{\blue Recalling the lift~\eqref{eqn:lifted-target-set} and \(p_0 = (0, u_*(x_0), x_0) \in \Manifold{L},\)
the first of these conditions is the controllability condition}
\begin{equation}
  \label{controllability}
  \tag{Con}
    \Ann\left(
      \TangentSpaceAt{\Manifold{L}}{p_0}
    \right)
    \cap
    \SpanInline{\Real}{
      \Codist{I}^{(n - n^*)}_{p_0},
      \D t_{p_0}
    }
    =
      \SpanInline{\Real}{ \D t_{p_0} }.
\end{equation}
The second is an involutivity condition demanding that on an open set \(\OpenSet{U}\subseteq \Manifold{M}\) containing \(p_0,\) for all \(p \in \OpenSet{U} \cap \Manifold{L},\)
\begin{equation}
  \label{involutivity}
  \tag{Inv}
  \Ann\left(\TangentSpaceAt{\Manifold{L}}{p}\right)
  \cap
  \SpanInline{\Real}{ \Codist{I}^{(k)}_p , \D t_{p} }
    \subseteq
      \GenerateInline{\Ideal{I}^{(k)}, \D t}^{(\infty)}_p,
  \qquad
    0 \leq k \leq n - n^*.
\end{equation}
Lastly, we require that the codistribution in the left-hand side of~\eqref{involutivity} satisfies, for all \(p \in \OpenSet{U} \cap \Manifold{L}\) and all \(0 \leq k \leq n - n^*,\)
\begin{equation}
  \label{constantdim}\blue
  \tag{Dim}
  \Dim(
    \Ann(
      \TangentSpaceAt{\Manifold{L}}{p_0}
    )
    \cap
    \SpanWord_\Real\{
      \Codist{I}^{(k)}_{p_0}, \D t_{p_0}
    \}
  )
    =
      \Dim(
        \Ann(
          \TangentSpaceAt{\Manifold{L}}{p}
        )
        \cap
        \SpanWord_\Real\{
          \Codist{I}^{(k)}_{p}, \D t_{p}
        \}
      ).
\end{equation}
This {\blue is} known as the constant dimensionality condition.
{\blue
The algorithm is used to prove that the conditions~\eqref{controllability},~\eqref{involutivity} and~\eqref{constantdim} imply the transverse feedback linearization problem is solvable at \(x_0 \in \Manifold{N}\) by explicitly constructing a transverse output.}
\begin{remark}
  Elements of other well known algorithms in~\cite{Gardner1992, Mullhaupt2006} as well as that of the ``Blended Algorithm'' of~\cite{Mehra2014} appear in the proposed algorithm.
\end{remark}
We are ready to state the main result of this article.
\begin{theorem}[Main Result]
  \label{thm:tfl}
  The system~\eqref{eqn:system} is locally transverse feedback linearizable with respect to the closed, embedded and controlled-invariant submanifold \(\Manifold{N} \subseteq \Real^n\) at \(x_0\) if, and only if,~\eqref{controllability} holds and there exists an open set \(\OpenSet{U} \subseteq \Manifold{M}\) of \(p_0\) on which conditions~\eqref{involutivity} and~\eqref{constantdim} hold.
\end{theorem}
There are a number of constants that appear in the results that follow.
First we define the indices\footnote{\blue The quotient is taken viewing the objects as \emph{real vector subspaces} of \(\CotangentSpaceAt{\Manifold{M}}{p_0}.\) It amounts to a difference in dimension.}, for all \(i \geq 0,\)
\begin{equation}
  \label{eqn:controlindex}
    \rho_i(p_0)
      \defineas
        \Dim
          \frac{
            \Ann(\TangentSpaceAt{\Manifold{L}}{p_0})
            \cap
            \SpanInline{\Real}{\Codist{I}^{(i)}_{p_0}, \D t_{p_0}}
          }{
            \Ann(\TangentSpaceAt{\Manifold{L}}{p_0})
            \cap
            \SpanInline{\Real}{\Codist{I}^{(i+1)}_{p_0}, \D t_{p_0}}
          }.
\end{equation}
Using the indices \(\rho_i\) define
\begin{equation}
  \label{eqn:controlindex:k}
  \kappa_i(p_0)
    \defineas
      \Card\{
        j
        \colon
        \rho_j(p_0) \geq i
      \},
  \qquad
  i \geq 0.
\end{equation}
We call \((\kappa_1, \ldots, \kappa_{n-n^*})\) the \emph{transverse controllability indices} of~\eqref{eqn:system} with respect to \(\Manifold{N}\) at \(x_0 = \pi(p_0)\)~\cite{Nielsen2008}.
Observe that, when the constant dimension condition~\eqref{constantdim} holds, \(\rho_i\) and \(\kappa_i\) are constant on an open set of \(\Manifold{N}\) containing \(p_0.\)
\begingroup
\blue
These indices play a role in the algorithm as they indicate in which ideals components of the transverse output appear.
It is fairly straightforward to show that the conditions are necessary for the transverse feedback linearization problem to be solvable.
As a result, we now briefly prove the necessity of the dual conditions.
\endgroup
\begin{proof}[Proof of Theorem~\ref{thm:tfl} (Necessity)]
  For this direction of the proof, we only provide a sketch;
  the proof is straightforward.
  Suppose that, on an open neighbourhood \(\OpenSet{V} \subseteq \Real^n\) of \(x_0,\) there exists a diffeomorphism \((\eta, \xi) \defineas \Phi(x)\) and feedback transformation \((v_\parallel, v_\pitchfork) \defineas \alpha(x) + \beta(x) u\) where, in the new coordinates, the nonlinear control system~\eqref{eqn:system} takes the form
  \begin{equation}\blue
    \label{main:nec:1}
  \begin{aligned}
    \dot{\eta}(t)
      &=
        \overline{f}(\eta(t), \xi(t))
        +
        \textstyle\sum_{j = 1}^{m - \rho_0}
          \overline{g}_{\parallel,j}(\eta(t), \xi(t))\,
          v_\parallel^j(t)
        +
        \textstyle\sum_{j = 1}^{\rho_0}
          \overline{g}_{\pitchfork,j}(\eta(t), \xi(t))\,
          v_\pitchfork^j(t),\\
    \dot{\xi}(t)
      &=
        A\, \xi(t) + \textstyle\sum_{j = 1}^{\rho_0} {b_j}\, v_\pitchfork^j(t),
  \end{aligned}
  \end{equation}
  and \((A,\) \(\begin{bsmallmatrix} b_1 & \cdots & b_{\rho_0}\end{bsmallmatrix})\) is in Brunovsk\'y normal form.
  In \((\eta, \xi)\)-coordinates the target set \(\Manifold{N}\) is locally
  \[
    \Phi(\OpenSet{V} \cap \Manifold{N}) = \left\{
      (\eta, \xi) \in \OpenSet{V}
      \colon
      \xi = 0
    \right\}.
  \]
  The lifted dynamical system~\eqref{eqn:systemOnM} on \(\Manifold{M}\) equivalent to~\eqref{main:nec:1} takes the form, on some open set \(\OpenSet{U} \subseteq \Manifold{M}\) containing \(p_0,\)
  \[
  \begin{aligned}\blue
    \dot{t} &= 1,\\
    \dot{\eta}
      &=
        \overline{f}(\eta, \xi)
        +
        \textstyle\sum_{j = 1}^{m - \rho_0}
          \overline{g}_{\parallel,j}(\eta, \xi)\,
          v_\parallel^j
        +
        \textstyle\sum_{j = 1}^{\rho_0}
          \overline{g}_{\pitchfork,j}(\eta, \xi)\,
          v_\pitchfork^j,\\
    \dot{\xi}
      &=
        A\, \xi + \textstyle\sum_{j = 1}^{\rho_0} {b_j}\, v_\pitchfork^j,
  \end{aligned}
  \]
  and the lifted manifold \(\Manifold{L} \subseteq \OpenSet{U}\) is locally
  \[
    \OpenSet{U}\cap \Manifold{L} = \left\{
      (t, v, \eta, \xi) \in \OpenSet{U}
      \colon
      t = \xi = 0
    \right\}.
  \]
  By~\cite[Lemma 4.3]{Nielsen2008}, the controllability indices of \((A,\) \(\begin{bsmallmatrix} b_1 & \cdots & b_{\rho_0}\end{bsmallmatrix})\) equal the transverse controllability indices of~\eqref{eqn:controlindex} so \((A,\) \(\begin{bsmallmatrix} b_1 & \cdots & b_{\rho_0}\end{bsmallmatrix})\) has \(\rho_0\) integration chains of length \(\kappa_1\) \(\geq\) \(\cdots\) \(\geq\) \(\kappa_{\rho_0}.\)
  We can therefore index the \(\xi\)-coordinates in the following way:
  fix \(1 \leq i \leq \rho_0.\)
  Write
  \[
  \begin{aligned}
    \dot{\xi}^{i, j}
      &=
        \xi^{i, j - 1}, \quad 2 \leq j \leq \kappa_i,\\
    \dot{\xi}^{i, 1}
      &=
        v^i_\pitchfork,\\
  \end{aligned}
  \]
  It is then clear that, for each \(p \in \OpenSet{U},\) we have the adapted basis structure,
  \[
    \Ann(\TangentSpaceAt{\Manifold{L}}{p})
    \cap
    \SpanInline{\Real}{\Codist{I}^{(k)}_p, \D t_p}
      =
        \SpanInline{\Real}{
          \D t_p
        }
        +
        \SpanInline{\Real}{
          \D \xi^{i, j}_p
          \colon
          1 \leq i \leq \rho_0,
          j > k
        }.
  \]
  From this we can deduce the TFL conditions.
  The constant dimension condition~\eqref{constantdim} follows directly.
  The controllability condition~\eqref{controllability} follows from considering the index \(k = n - n^*.\)
  Observe there are no \(\xi^{i, j}\) with \(j > n - n^*\) since that would imply the existence of more than \(n - n^*\) transverse directions to \(\Manifold{N}.\)
  Therefore
  \[
    \Ann(\TangentSpaceAt{\Manifold{L}}{p})
    \cap
    \SpanInline{\Real}{\Codist{I}^{(n - n^*)}_p, \D t_p}
      =
        \SpanInline{\Real}{
          \D t_p
        }.
  \]
  The involutivity condition~\eqref{involutivity} follows because we have exact generators \(\D t\) and \(\D \xi^{i, j}\) that generate the codistribution
  \(
    \Ann(\TangentSpaceAt{\Manifold{L}}{p})
    \cap
    \SpanInline{\Real}{\Codist{I}^{(k)}_p, \D t_p}
  \)
  for all \(0 \leq k \leq n - n^*.\)
\end{proof}
\begingroup\blue
The proof of sufficiency for Theorem~\ref{thm:tfl} is constructive and results in the algorithm presented in Section~\ref{sec:mainproof}.
\endgroup

\section{Example: Checking the Conditions}
\label{sec:example:conditions}
Consider the nonlinear, control-affine system
\begin{equation}
  \label{eqn:example:1}
  \begin{alignedat}{5}
    \dot{x}^1 &= -x^2 - x^2\, u^2 &
      &\qquad &
      \dot{x}^4 &= u^1
      &\qquad &
      \dot{x}^7 &= x^5 + x^1\, u^2\\
    \dot{x}^2 &= x^1 &
      &\qquad &
      \dot{x}^5 &= x^6 - x^1\, u^2
      &\qquad &
      &\qquad \\
    \dot{x}^3 &= x^3\, x^4 + x^3\, u^1 &
      &\qquad &
      \dot{x}^6 &= x^7 + x^6 - x^3\, x^5 + x^1\, u^2
      &\qquad &
      &\qquad \\
  \end{alignedat}
\end{equation}
and the closed, embedded \(2\)-dimensional submanifold
\begin{equation}
  \label{eqn:example:N}
  \Manifold{N}
    \defineas
      \left\{
        x \in \Real^7
        \colon
        (x^1)^2 + (x^2)^2 - x^3 = x^4 = x^5 = x^6 = x^7 = 0
      \right\},
\end{equation}
rendered controlled-invariant by \(u_*(x) = 0.\)
Fix a point \(x_0 = (2, 0, 4, 0, 0, 0, 0) \in \Manifold{N}.\)
In light of Theorem~\ref{thm:LTFL-Nielsen-3.1}, if~\eqref{eqn:example:1} is transverse feedback linearizable with respect to \(\Manifold{N}\) at \(x_0,\) then there exists either a single scalar function yielding a relative degree \(5\) at \(x_0\) that vanishes on \(\Manifold{N}\) or two scalar functions yielding a vector relative degree \((\kappa_1, 5 - \kappa_1)\) at \(x_0\) that simultaneously vanish on \(\Manifold{N}.\)
Natural candidates can be picked out of the functions that define \(\Manifold{N}\) since they satisfy~\ref{thm:LTFL-Nielsen-3.1:1} of Theorem~\ref{thm:LTFL-Nielsen-3.1}.
Unfortunately, all of the scalar functions used to define \(\Manifold{N}\) in~\eqref{eqn:example:N} either yield a relative degree of \(1\) at \(x_0\) or do not yield a relative degree at all.
As a result, we cannot directly use them to form an output that satisfies~\ref{thm:LTFL-Nielsen-3.1:2} of Theorem~\ref{thm:LTFL-Nielsen-3.1}.

However, as we now show using our dual TFL conditions, the transverse feedback linearization problem is solvable at \(x_0 \in \Manifold{N}.\)
The ideals \(\Ideal{I}^{(i)}\) in~\eqref{eqn:flag} for system~\eqref{eqn:example:1} {\blue are such that
\[
\begin{aligned}
  \GenerateInline{\Ideal{I}^{(0)}, \D t}
    &=
      \GenerateInline{
        \D x^1, \D x^2, \D x^3, \D x^4, \D x^5, \D x^6, \D x^7, \D t
      }\\
  \GenerateInline{\Ideal{I}^{(1)}, \D t}
    &=
      \GenerateInline{
        x^1\,\D x^1 + x^2\,\D x^7,
        \D x^2 - \D x^7,
        \D x^3 - x^3\,\D x^4,
        \D x^5 + \D x^7,
        \D x^6 - \D x^7,
        \D t
      }\\
  \GenerateInline{\Ideal{I}^{(2)}, \D t}
    &=
      \GenerateInline{
        \beta^1, \beta^2, \D t
      },\\ 
  \GenerateInline{\Ideal{I}^{(3)}, \D t}
    &=
      \GenerateInline{ \D t },
\end{aligned}
\]}%
where \(\beta^1,\) \(\beta^2 \in \DifferentialForms{\Manifold{M}}\) are smooth one-forms whose expressions we omit for clarity.
The derived flag~\eqref{eqn:flag} terminates at \(\Ideal{I}^{(3)}.\)
Immediately we see that the controllability condition~\eqref{controllability} holds since
\[
  \GenerateInline{ \Ideal{I}^{(n - n^*)}, \D t }
  =
  \GenerateInline{ \Ideal{I}^{(5)}, \D t }
  =
  \GenerateInline{ \Ideal{I}^{(3)}, \D t }
  =
  \GenerateInline{ \D t }.
\]
Next we check the constant dimensionality condition~\eqref{constantdim}.
Observe that, for any \(p \in \Manifold{L}\) in a sufficiently small open set containing \(p_0,\)
\begin{equation}
  \label{eqn:example:flag}\blue
\begin{aligned}
  \Ann(\TangentSpaceAt{\Manifold{L}}{p})
  \cap
  \SpanInline{\Real}{\Codist{I}^{(0)}_p, \D t_p}
    &=
      \Ann(\TangentSpaceAt{\Manifold{L}}{p}),\\
  \Ann(\TangentSpaceAt{\Manifold{L}}{p})
  \cap
  \SpanInline{\Real}{\Codist{I}^{(1)}_p, \D t_p}
    &=
      \SpanWord_\Real\{
        \D x^6 - \D x^7,
        \D x^5 + \D x^7,\\
        -2 x^1\,\D x^1 &- 2 x^2\, \D x^2 + \D x^3 - x^3\,\D x^4 - 2
        x^2\,\D x^7, \D t\},\\
  \Ann(\TangentSpaceAt{\Manifold{L}}{p})
  \cap
  \SpanInline{\Real}{\Codist{I}^{(2)}_p, \D t_p}
  &=
    \SpanInline{\Real}{
      \D x^5 + \D x^7,
      \D t
    },\\
  \Ann(\TangentSpaceAt{\Manifold{L}}{p})
  \cap
  \SpanInline{\Real}{\Codist{I}^{(3)}_p, \D t_p}
  &=
    \SpanInline{\Real}{
      \D t
    }.\\
\end{aligned}
\end{equation}
Thus the constant dimensionality condition~\eqref{constantdim} holds.
It remains to check the involutivity condition~\eqref{involutivity} holds.
Note that the ideals \(\langle\Ideal{I}^{(0)}, \D t\rangle,\) \(\langle \Ideal{I}^{(1)}, \D t\rangle\) and \(\langle \Ideal{I}^{(3)}, \D t\rangle\) are all differential ideals.
Therefore, it suffices to check that
\[
  \Ann(\TangentSpaceAt{\Manifold{L}}{p})
  \cap
  \SpanInline{\Real}{\Codist{I}^{(2)}_p, \D t_p}
  \subseteq
  \GenerateInline{
    \Ideal{I}^{(2)}, \D t
  }_p^{(\infty)}.
\]
Using \texttt{Maple}, we directly compute the derived flag for \(\langle \Ideal{I}^{(2)}, \D t \rangle\) and verify it converges to
\[
  \GenerateInline{
    \Ideal{I}^{(2)}, \D t
  }^{(\infty)}
    =
      \GenerateInline{
        \D x^5 + \D x^7, \D t
      }.
\]
Comparing this with the expression for
\(
  \Ann(\TangentSpaceAt{\Manifold{L}}{p})
  \cap
  \SpanWord\{\Codist{I}^{(2)}_p, \D t_p\}
\)
above,
we see that the involutivity condition~\eqref{involutivity} holds.
As a result, the transverse feedback linearization problem is solvable for~\eqref{eqn:example:1} with respect to \(\Manifold{N}\) at \(x_0.\)
In Section~\ref{sec:example:algorithm} we revisit this example and explicitly use the proposed algorithm to produce a non-trivial transverse output with respect to \(\Manifold{N}\) at \(x_0\) that solves the TFL problem.

\section{Supporting Results}
\label{sec:tech}
Before completing the proof of the main result, we present a number of supporting propositions.
% They are organized at a high-level by the problem they tackle.
% Section~\ref{sec:tech:transverseoutput} presents a series of technical results that connect the conditions of transverse feedback linearization ---~\eqref{controllability},~\eqref{involutivity},~\eqref{constantdim} --- to the algorithmic procedure.
They will connect the conditions of transverse feedback linearization ---~\eqref{controllability},~\eqref{involutivity},~\eqref{constantdim} --- to the algorithmic procedure.

% \subsection{In Preparation to Adapt to \texorpdfstring{\(\Manifold{L}\)}{L} and the System Structure}
\label{sec:tech:transverseoutput}
Recall Proposition~\ref{prop:relativedegree} established that an output \(h\) yielding a uniform vector relative degree at \(p_0\) must have a differential that lives in a specific ideal but not live in the subsequent ideal of the derived flag~\eqref{eqn:flag}.
Now we ask a different question:
Given an output that yields a uniform vector relative degree, how do we find new scalar outputs that (1) yield a smaller uniform vector relative degree and (2) combines with previously known scalar outputs to yield a vector relative degree?
All the while, we must ensure that (3) the new outputs vanish on the target manifold \(\Manifold{N}.\)

Points (1) and (2) are classically performed by ``adapting'' the basis of \emph{exact} generators for the derived flag~\eqref{eqn:flag}.
It is point (3) that imposes a greater degree of difficulty in the adaptation process precisely because we ask that the differentials of the outputs \(\D h^i\) annihilate tangent vectors to \(\Manifold{N}.\)
At a high level, we present three procedures that together will be used to correctly adapt the derived flag.
These are:
\begin{enumerate}[label={(\alph*)}]
  \item{
    \label{adapt:a}
    finding generators that ``drop off'' when computing the derived flag~\eqref{eqn:flag}, 
  }
  \item{
    \label{adapt:b}
    grouping generators into those that annihilate tangent vectors to \(\Manifold{L}\) and those that do not, and
  }
  \item{
    \label{adapt:c}
    rewriting generators so that the induced output has
    a full rank decoupling matrix.
  }
\end{enumerate}
{\blue Note that these steps are performed repeatedly throughout the algorithm.}
This section presents technical results that demonstrate how to perform steps~\ref{adapt:a} and ~\ref{adapt:b}.
Step~\ref{adapt:c} is presented in the proof of the main result.

Before discussing these subprocedures in any detail, we must know the dimension of the subspace in an ideal of the derived flag~\eqref{eqn:flag} that annihilates tangent vectors to \(\Manifold{L}.\)
The first lemma shows how the controllability condition~\eqref{controllability} determines this.
\begin{lemma}
  \label{lemma:1}
  If~\eqref{controllability}, then for all \(0 \leq k \leq n - n^*\)
  \[
    \Dim(
      \Ann(\TangentSpaceAt{\Manifold{L}}{p_0})
      \cap
      \SpanInline{\Real}{ \Codist{I}^{(k)}_{p_0}, \D t_{p_0}}
    )
    =
    1 + \textstyle\sum_{i = k}^{n - n^* - 1} \rho_{i},
  \]
  and \(n - n^* = \sum_{i = 0}^{n - n^* - 1}\rho_i.\)
\end{lemma}
\begin{proof}
  By~\eqref{controllability} we have
  \[
    \Dim(
      \Ann(\TangentSpaceAt{\Manifold{L}}{p_0})
      \cap
      \SpanInline{\Real}{ \Codist{I}^{(n - n^*)}_{p_0}, \D t_{p_0}}
    )
      =
        1
  \]
  so the formula holds for \(k = n - n^*.\)
  Suppose, by way of induction, that for some \(1 \leq k \leq n - n^*\)
  \begin{equation}
    \label{eqn:lemma:1:1}
    \Dim(
      \Ann(\TangentSpaceAt{\Manifold{L}}{p_0})
      \cap
      \SpanInline{\Real}{
        \Codist{I}^{(k)}_{p_0}, \D t_{p_0}
      }
    )
      =
        1 + \textstyle\sum_{i = k}^{n - n^* - 1} \rho_{i}.
  \end{equation}
  Consider the left-hand side of~\eqref{eqn:lemma:1:1}.
  By~\eqref{eqn:controlindex}
  \[
    \Dim(
      \Ann(\TangentSpaceAt{\Manifold{L}}{p_0})
      \cap
      \SpanWord_\Real\{
        \Codist{I}^{(k-1)}_{p_0},
        \D t_{p_0}
      \}
    )
      =
      \Dim(
        \Ann(\TangentSpaceAt{\Manifold{L}}{p_0})
        \cap
        \SpanWord_\Real\{
          \Codist{I}^{(k)}_{p_0},
          \D t_{p_0}
        \}
      )
      +
      \rho_{k-1}.
  \]
  Apply the inductive hypothesis~\eqref{eqn:lemma:1:1} and conclude
  \[
    \Dim(
      \Ann(\TangentSpaceAt{\Manifold{L}}{p_0})
      \cap
      \SpanInline{\Real}{
        \Codist{I}^{(k-1)}_{p_0},
        \D t_{p_0}
      }
    )
      =
        1
        +
        \textstyle\sum_{i = k - 1}^{n - n^* - 1} \rho_{i}.
  \]
  We now verify the final fact.
  First observe
  \[
    \SpanInline{\Real}{\Codist{I}^{(0)}_{p_0}, \D t_{p_0}}
    =
    \SpanInline{\Real}{\D x^1_{p_0}, \ldots, \D x^n_{p_0}, \D t_{p_0}}.
  \]
  By construction of \(\Manifold{L}\) we have
  \[
    \Ann(\TangentSpaceAt{\Manifold{L}}{p_0})
    \subseteq
    \SpanInline{\Real}{\D x^1_{p_0}, \ldots, \D x^n_{p_0}, \D t_{p_0}}.
  \]
  Therefore
  \[
    \Dim(
      \Ann(\TangentSpaceAt{\Manifold{L}}{p_0})
      \cap
      \SpanInline{\Real}{\Codist{I}^{(0)}_{p_0}, \D t_{p_0}}
    )
      =
        \Dim(
          \Ann(\TangentSpaceAt{\Manifold{L}}{p_0})
        )
      =
        1 + n - n^*.
  \]
  Combining this with the formula
  \[
    \Dim(
      \Ann(\TangentSpaceAt{\Manifold{L}}{p_0})
      \cap
      \SpanInline{\Real}{\Codist{I}^{(0)}_{p_0}, \D t_{p_0}}
    )
      =
        1 + \textstyle\sum_{i = 0}^{n - n^* - 1} \rho_i,
  \]
  completes the proof.
\end{proof}
The previous result concerned the ideal but not its differential closure.
Next we show that, when the involutivity condition holds, we can work with either the ideal or its differential closure as long as we are only concerned with the differentials that annihilate tangent vectors to \(\Manifold{L}.\)
\begin{proposition}
  \label{prop:invequals}
  If~\eqref{involutivity} holds on some open set \(\OpenSet{U}\) containing \(p_0,\) then for all \(p \in \OpenSet{U} \cap \Manifold{L}\)
  \[
    \Ann(\TangentSpaceAt{\Manifold{L}}{p})
    \cap
    \SpanInline{\Real}{\Codist{I}^{(k)}_{p}, \D t_p}
    =
    \Ann(\TangentSpaceAt{\Manifold{L}}{p})
    \cap
    \GenerateInline{\Ideal{I}^{(k)}, \D t}^{(\infty)}_p,
    \quad 0 \leq k \leq n - n^*.
  \]
\end{proposition}
\begin{proof}
  Fix \(p \in \OpenSet{U} \cap \Manifold{L}\) and intersect both sides of the involutivity condition~\eqref{involutivity} with \(\Ann(\TangentSpaceAt{\Manifold{L}}{p}).\)
  Use the fact that \(\GenerateInline{\Ideal{I}^{(k)}, \D t}^{(\infty)} \subseteq \langle \Ideal{I}^{(k)}, \D t\rangle\) to arrive at the equality.
\end{proof}
Up until this point, we have concerned ourselves with ideals in the flag~\eqref{eqn:flag} up to index \(n-n^*.\)
This is not necessary.
Observe that
\begin{equation}
  \label{eqn:rho-to-nstar}
  \rho_{\kappa_1} = \cdots = \rho_{n-n^*} = 0,
\end{equation}
by definition.
If the controllability condition~\eqref{controllability} holds, then it holds with \(n-n^*\) replaced by \(\kappa_1.\)
As a result, it suffices to look at the ideals in the flag~\eqref{eqn:flag} with indices up to and including \(\kappa_1.\)

As has already been mentioned, finding outputs of vector relative degree amounts to finding an appropriately adapted basis for the derived flag.
Recall that, when \(\Manifold{N} = \{ x_0 \},\) the transverse feedback linearization problem reduces to the exact state-space feedback linearization problem~\cite{Nielsen2008}.
As a result, we expect our algorithm to apply just as well to feedback linearization.
In the exact feedback linearization algorithm presented in~\cite{Gardner1992}, the generators for the differential ideals \(\langle \Ideal{I}^{(k)}, \D t\rangle^{(\infty)}\) are assumed to satisfy,
\begin{equation}
  \label{eqn:gsadapted}\blue
\begin{alignedat}{2}
  &\scriptstyle\langle\Ideal{I}^{(0)}, \D t\rangle^{(\infty)}&
    &{\scriptstyle=}
    \scriptstyle\langle\omega^1, \ldots, \ldots, \omega^n, \D t\rangle,\\
  &\scriptstyle\langle\Ideal{I}^{(1)}, \D t\rangle^{(\infty)}&
    &{\scriptstyle=}
    \scriptstyle\langle\omega^1, \ldots, \omega^{n - \rho_0}, \D t\rangle,\\
  & &
    &\phantom{'}{\scriptstyle\vdots}\\
  &\scriptstyle\langle\Ideal{I}^{(\kappa_1-1)}, \D t\rangle^{(\infty)}&
    &{\scriptstyle=}
    \scriptstyle\langle\omega^1, \omega^{2}, \D t\rangle,\\
  &\scriptstyle\langle\Ideal{I}^{(\kappa_1)}, \D t\rangle^{(\infty)}&
    &{\scriptstyle=}
    \scriptstyle\langle\D t\rangle.
\end{alignedat}
\end{equation}
The generators ``drop off'' as the derived flag is computed.
This is precisely what is meant by subprocedure~\ref{adapt:a}.
Recall, however, that we must not only find exact generators for the differential ideal, but ones that annihilate tangent vectors to \(\Manifold{L}.\)
Unfortunately, the process of finding these annihilating exact one-forms requires rewriting the generators of the ideal \(\langle \Ideal{I}^{(k)}, \D t\rangle\) which would ruin the adapted structure~\eqref{eqn:gsadapted}.
This is why we explicitly demonstrate that the re-adaptation process, subprocedure~\ref{adapt:a} in particular, can be applied when required.

Unfortunately, finding exact generators that annihilate tangent vectors to \(\Manifold{L}\) is itself challenging.
As seen in~\cite[Remark 2]{DSouza2021}, constructing annihilating one-forms directly from known exact forms can ruin their exactness.
The remark suggests solving another Cauchy problem to find the correct exact one-form.
This is tenable in the single-input case, but is not a satisfying solution in the multi-input case especially when seeking multiple, independent, scalar, transverse outputs.
To avoid this additional integration step, we introduce maps \(\blue H_k\) whose differential is the exact generator for \(\langle \Ideal{I}^{(k)}, \D t\rangle^{(\infty)}.\)
The re-adaptation process mentioned earlier, subprocedure~\ref{adapt:a}, amounts to rewriting the components of these smooth maps so that the image of the pullback \(\blue\Image (\J H_k)^*\) is preserved.

\begin{remark}
  \label{rem:WhyIntegrals}
  Manipulating the integrals of the exact generators suggests integrating all the Frobenius systems in the derived flag~\eqref{eqn:flag}.
  In Section~\ref{sec:example:algorithm}, we show that integration is only required for the differential ideals \(\langle \Ideal{I}^{(k)}, \D t\rangle^{(\infty)}\) at which \(k\) is a \emph{distinct} transverse controllability index.
\end{remark}

We start by showing the existence of the aforementioned maps.
The TFL conditions ---~\eqref{controllability}, \eqref{involutivity} and~\eqref{constantdim} --- allow us to construct the map \(\blue H_k\) explicitly with a specific rank deficiency when restricted to \(\Manifold{L}.\)
This deficiency will ultimately be used to construct the transverse outputs.
\begin{lemma}
  \label{lemma:3}
  If~\eqref{controllability} holds and there exists an open set \(\OpenSet{U}\) containing \(p_0\) where \eqref{involutivity} and~\eqref{constantdim} hold, then for every \(0 \leq k \leq n - n^*\) there exists an integer \(\ell_k \geq 1 + \sum_{i = k}^{n - n^* - 1} \rho_{i},\) a possibly smaller open neighbourhood \(\OpenSet{V}\subseteq \OpenSet{U}\) containing \(p_0,\) and a smooth map \(\blue H_k: \OpenSet{V} \to \Real^{\ell_k}\) satisfying the characteristic property
  \begin{equation}\blue
    \label{lemma:2:a}
    \GenerateInline{ \D H_k^1, \ldots, \D H_k^{\ell_k} }
      =
        \GenerateInline{\Ideal{I}^{(k)}, \D t}^{(\infty)},
  \end{equation}
  with constant rank, on \(\OpenSet{V}\cap \Manifold{L},\) equal to
  \[\blue
    \Rank \left.H_k\right|_{\OpenSet{V} \cap \Manifold{L}}
      =
        \ell_k - (1 + \textstyle\sum_{i = k}^{n - n^*-1} \rho_{i}).
  \]
\end{lemma}
\begin{proof}
  The differential ideal \(\langle \Ideal{I}^{(k)}, \D t\rangle^{(\infty)}\) is simply, finitely generated so, by Frobenius's Theorem, there exists \(\ell_k\) exact generators \(\D w^1,\) \(\ldots,\) \(\D w^{\ell_k}\) on some open neighbourhood \(\OpenSet{V} \subseteq \OpenSet{U}\) of \(p_0.\)
  Define a smooth map \(\blue H_k: \OpenSet{V} \to \Real^{\ell_k}\) by
  \[\blue
    H_k(p) \defineas \left( w^1(p), \ldots, w^{{\ell_k}-1}(p), t \right).
  \]
  Without loss of generality, take \(w^i\) to be smooth functions of the state.
  By construction, \(H_k\) satisfies the characteristic property~\eqref{lemma:2:a}.

  We already know, by Assumption~\ref{assume:regular}, that \(H_k\) has constant rank.
  It is not directly obvious that \(\left.H_k\right|_{\OpenSet{V}\cap \Manifold{L}}\) has constant rank as well.
  Since~\eqref{involutivity} holds over \(\OpenSet{V},\) invoke Proposition~\ref{prop:invequals} to find
  \begin{equation}
    \label{eqn:lemma:2:1}
    \Ann(\TangentSpaceAt{\Manifold{L}}{p})
    \cap
    \GenerateInline{\Ideal{I}^{(k)}, \D t}^{(\infty)}_p
    =
    \Ann(\TangentSpaceAt{\Manifold{L}}{p})
    \cap
    \SpanInline{\Real}{\Codist{I}^{(k)}_{p}, \D t_p},
  \end{equation}
  for all \(p \in \OpenSet{V} \cap \Manifold{L}.\)
  It then follows by~\eqref{constantdim} that
  \[
    \Dim\left(
      \Ann(\TangentSpaceAt{\Manifold{L}}{p})
      \cap
      \GenerateInline{\Ideal{I}^{(k)}, \D t}^{(\infty)}_p
    \right)
      =
        \text{constant}.
  \]
  Then use the characteristic property~\eqref{lemma:2:a} to determine that \(\left. H_k \right|_{\OpenSet{V} \cap \Manifold{L}}\) must have constant rank.

  We now proceed by directly computing its rank at a point \(p \in \OpenSet{V} \cap \Manifold{L}.\)
  Because the rank of \(\left.H_k\right|_{\OpenSet{V} \cap \Manifold{L}}\) is constant, it suffices to compute its rank at \(p_0.\)
  Compute the dimension on both sides of~\eqref{eqn:lemma:2:1} and invoke Lemma~\ref{lemma:1} to find,
  \[
    \Dim\left( 
      \Ann(\TangentSpaceAt{\Manifold{L}}{p_0})
      \cap
      \GenerateInline{\Ideal{I}^{(k)}, \D t}^{(\infty)}_{p_0}
     \right)
      =
        1 + \textstyle\sum_{i = k}^{n-n^*-1} \rho_i.
  \]
  It immediately follows that
  \[
    \Rank \left.H_k\right|_{\OpenSet{V} \cap \Manifold{L}}
      =
        \ell_k - ( 1 + \textstyle\sum_{i=k}^{n-n^*-1} \rho_i )
  \]
\end{proof}
The proof of Lemma~\ref{lemma:3} did not specifically rely on the way the map \(H_k\) was constructed (as the integral of a Frobenius system).
The proof holds without modification for any map \(H_k\) that satisfies the characteristic property~\eqref{lemma:2:a}.
The next corollary states this fact.
\begin{corollary}
  Suppose~\eqref{controllability} holds and there exists an open set \(\OpenSet{U}\) containing \(p_0\) where \eqref{involutivity} and~\eqref{constantdim} hold.
  If a smooth map \(H_k: \OpenSet{U} \to \Real^{\ell_k}\) satisfies the characteristic property~\eqref{lemma:2:a} then it has constant rank on \(\OpenSet{U}\cap \Manifold{L}\) equal to
  \[
    \Rank \left.H_k\right|_{\OpenSet{U} \cap \Manifold{L}}
      =
      {\ell_k} - (1 + \textstyle\sum_{i = k}^{n - n^* - 1} \rho_{i}).
  \]
\end{corollary}
We now ask whether the \(H_k\) can be ``adapted'' to include components that are constant (w.l.o.g. zero) on \(\Manifold{L}.\)
In general, level sets of the form \(H^i_k(p) = H^i_k(p_0)\) do not contain \(\Manifold{L},\) as depicted in Figure~\ref{fig:AdaptToL:b}.
Figure~\ref{fig:AdaptToL:a} shows how the corresponding generators for the codistribution \(\langle \Ideal{I}^{(k)}, \D t\rangle^{(\infty)}\) do not annihilate vectors tangent to \(\Manifold{L}\)  although the codistribution has a non-trivial intersection with \(\Ann(\TangentSpaceAt{\Manifold{L}}{p_0}).\)
We can use the rank deficiency of \(H_k\) on \(\Manifold{L}\) to construct an adaptation of \(H_k\) where the leading  \(1 + \sum_{i = k}^{n - n^* - 1} \rho_{i}\) components have level sets that locally contain \(\Manifold{L}.\)
Figure~\ref{fig:AdaptToL:c} shows the level sets of the newly rewritten \(\tilde{H}_k.\)
At a point \(p \in \Manifold{L},\) the leading components' differential lives in \(\Ann(\TangentSpaceAt{\Manifold{L}}{p})\) \(\cap\) \(\langle \Ideal{I}^{(k)}, \D t\rangle^{(\infty)}_{p}\) as seen in Figure~\ref{fig:AdaptToL:d}.
\begin{lemma}
  \label{lem:adaptToL}
  If \(H_k: \OpenSet{V} \to \Real^{\ell_k}\) is a smooth map satisfying the characteristic property~\eqref{lemma:2:a} and its restriction has constant rank equal to
  \[
    \Rank \left.H_k\right|_{\OpenSet{V} \cap \Manifold{L}} 
      =
        \ell_k - \rho,
  \]
  then there exists a new map \(\tilde{H}_k: \OpenSet{V} \to \Real^{\ell_k}\) which satisfies the characteristic property~\eqref{lemma:2:a} and, for all \(p \in \OpenSet{V}\cap \Manifold{L},\) \(v_p \in \TangentSpaceAt{\Manifold{L}}{p},\) \(i \in \{1,\ldots,\rho\},\) \(\D \tilde{H}_k^i(v_p) = 0.\)
\end{lemma}
\begin{proof}
  Apply~\cite[Rank Theorem (Proposition 4.12)]{Lee2012} and shrink \(\OpenSet{V}\) if necessary to find a coordinate chart \(\varphi\) for \(\Manifold{L}\) and a coordinate chart \(\psi\) for \(\Real^{\ell_{k}}\) so the composition \(\psi \circ \left. H_{k} \right|_{\OpenSet{V} \cap \Manifold{L}} \circ \varphi^{-1}\) takes the form
  \[
    \psi \circ \left.H_{k}\right|_{\OpenSet{V} \cap \Manifold{L}} \circ \varphi^{-1}
    =
      (
        0,\;\ldots,\;0,
        \star,\;\ldots,\;\star
      ),
  \]
  with \(\rho\) leading zeros.
  Define \(\tilde{H}_{k}\defineas \psi \circ H_{k}.\)
  The map \(\tilde{H}_{k}\) still satisfies the characteristic property~\eqref{lemma:2:a} but is now ``adapted'' so that the first \(\rho\) components vanish on \(\Manifold{L}\) and, consequently, their differentials must annihilate tangent vectors to \(\Manifold{L}\) as required.
\end{proof}
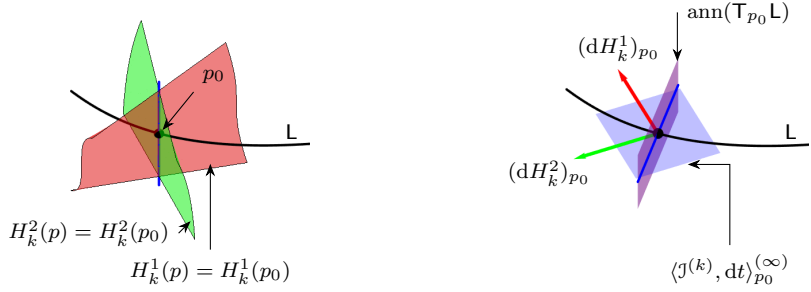
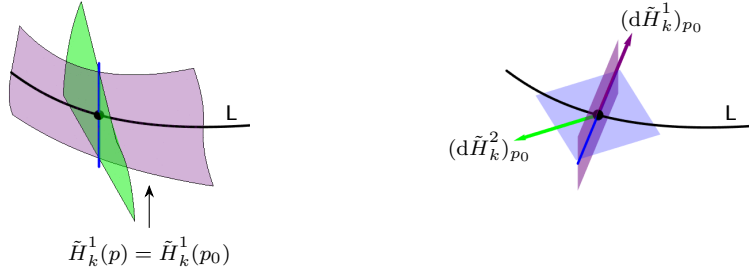
\begin{figure}
  \centering
  \begin{subfigure}[t]{0.49\textwidth}
    \centering
    \begin{tikzpicture}[x=1in, y=1in, scale=0.9]
      \node[draw=none, fill=none, anchor={south west}] at (0,0) {%
        \includegraphics[scale=0.9]{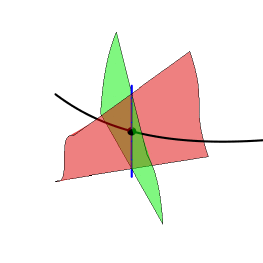}%
      };
      \node[at={(1.7, 0.9)}, anchor={base}, font={\footnotesize}] (lblL) {%
        \(\Manifold{L}\)%
      };
      \node[at={(1.25, 1.25)}, anchor={base}, font={\footnotesize}] (lblp0) {%
        \(p_0\)%
      };
      \node[at={(1.75, 0)}, anchor={south east}, font={\footnotesize}] (lblF1) {%
        \(H_k^{1}(p) = H_k^1(p_0)\)%
      };
      \node[at={(0, 0.35)}, anchor={west}, font={\footnotesize}] (lblF2) {%
        \(H_k^{2}(p) = H_k^2(p_0)\)%
      };
      \draw[-{Stealth}]
        (lblp0.south west) -- (0.95, 0.95);
      \draw[-{Stealth}]
        (lblF2.east) --+ (0.07, 0.07);
      \draw[-{Stealth}]
        (lblF1.north) --+ (0, 0.5);
    \end{tikzpicture}
    \caption{%
      The intersection of the level sets of \(H_k^i\) form an integral submanifold (light blue) of \(\langle \Ideal{I}^{(k)}, \D t\rangle^{(\infty)}\) passing through \(p_0.\)
      Neither level set contains \(\Manifold{L}.\)
    }
    \label{fig:AdaptToL:b}
  \end{subfigure}
  \hfill
  \begin{subfigure}[t]{0.49\textwidth}
    \centering
    \begin{tikzpicture}[x=1in, y=1in, scale=0.9]
      \node[draw=none, fill=none, anchor={south west}] at (0,0) {%
        \includegraphics[scale=0.9]{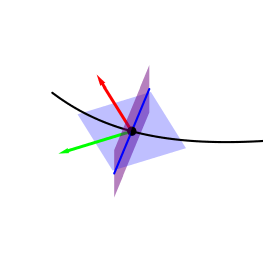}%
      };
      
      \node[at={(1.7, 0.9)}, anchor={base}, font={\footnotesize}] (lblL) {%
        \(\Manifold{L}\)%
      };
      \node[at={(1.75, 0)}, anchor={south east}, font={\footnotesize}] (lblCodist) {%
        \(\langle \Ideal{I}^{(k)}, \D t\rangle^{(\infty)}_{p_0}\)%
      };
      \draw[-{Stealth}]
        (lblCodist.north) |- (1.1, 0.75);

      \node[at={(1.75, 1.75)}, anchor={north east}, font={\footnotesize}] (lblAnn) {%
        \(\Ann(\TangentSpaceAt{\Manifold{L}}{p_0})\)%
      };
      \draw[-{Stealth}]
        (lblAnn.west) -| (1.04, 1.35);
      
      \node[at={(0.7, 1.3)}, anchor={south}, font={\footnotesize}] {%
        \((\D H_k^{1})_{p_0}\)%
      };
      \node[at={(0, 0.70)}, anchor={west}, font={\footnotesize}] {%
        \((\D H_k^{2})_{p_0}\)%
      };
    \end{tikzpicture}
    \caption{%
      The original choice of generators do not annihilate vectors tangent to \(\Manifold{L}.\)
    }
    \label{fig:AdaptToL:a}
  \end{subfigure}  
  \\
  \begin{subfigure}[t]{0.49\textwidth}
    \centering
    \begin{tikzpicture}[x=1in, y=1in, scale=0.9]
      \node[draw=none, fill=none, anchor={south west}] at (0,0) {%
        \includegraphics[scale=0.9]{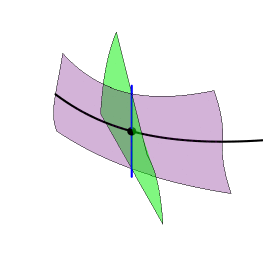}%
      };
      \node[at={(1.7, 0.9)}, anchor={base}, font={\footnotesize}] (lblL) {%
        \(\Manifold{L}\)%
      };
      \node[at={(1.75, 0)}, anchor={south east}, font={\footnotesize}] (lblF1) {%
        \(\tilde{H}_k^{1}(p) = \tilde{H}_k^{1}(p_0)\)%
      };
      \draw[-{Stealth}]
        (lblF1.north) --+ (0, 0.25);
    \end{tikzpicture}
    \caption{%
      There exists a map \(\tilde{H}_k\) so that the zero locus of the leading components contain \(\Manifold{L}\) while preserving the integral submanifold.
    }
    \label{fig:AdaptToL:c}
  \end{subfigure}
  \hfill
  \begin{subfigure}[t]{0.49\textwidth}
    \centering
    \begin{tikzpicture}[x=1in, y=1in, scale=0.9]
      \node[draw=none, fill=none, anchor={south west}] at (0,0) {%
        \includegraphics[scale=0.9]{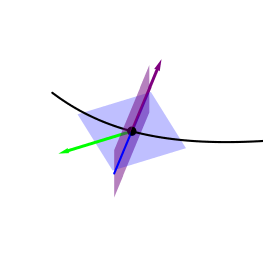}%
      };
      \node[at={(1.7, 0.9)}, anchor={base}, font={\footnotesize}] (lblL) {%
        \(\Manifold{L}\)%
      };
      \node[at={(1.3, 1.45)}, anchor={base}, font={\footnotesize}] {%
        \((\D \tilde{H}_k^1)_{p_0}\)%
      };
      \node[at={(0, 0.75)}, anchor={west}, font={\footnotesize}] {%
        \((\D \tilde{H}_k^{2})_{p_0}\)%
      };
    \end{tikzpicture}
    \caption{%
      The smooth one-form \(\D \tilde{H}^{1}\) lives in \(\langle \Ideal{I}^{(k)}, \D t\rangle^{(\infty)}\) and annihilates the tangent space of \(\Manifold{L}.\)
    }
    \label{fig:AdaptToL:d}
  \end{subfigure}
  \caption{%
    An arbitrary set of generators (red, green) for the codistribution \(\langle \Ideal{I}^{(k)}, \D t\rangle^{(\infty)}\) (blue) is adapted to annihilate the tangent space of \(\Manifold{L}\) (black).
  }
  \label{fig:AdaptToL}
\end{figure}
The established facts ensure that, assuming~\eqref{controllability} holds at \(p_0\) and there exists an open set \(\OpenSet{U}\) so that~\eqref{involutivity} and~\eqref{constantdim} hold, we can write, for all \(p \in \OpenSet{U} \cap \Manifold{L},\)
\begin{equation}
  \label{eqn:unadapted}\blue
\begin{alignedat}{2}
  &\scriptstyle\Ann(\TangentSpaceAt{\Manifold{L}}{p})
  \cap
  \langle\Ideal{I}^{(0)}, \D t\rangle^{(\infty)}_p&
    &{\scriptstyle=}
    \scriptstyle\SpanWord\{\D H_0^1, \ldots, \ldots, \ldots, \D H_0^{n - n^*}, \D t\},\\
  &\scriptstyle\Ann(\TangentSpaceAt{\Manifold{L}}{p})
  \cap
  \langle\Ideal{I}^{(1)}, \D t\rangle^{(\infty)}_p&
    &{\scriptstyle=}
    \scriptstyle\SpanWord\{\D H_1^1, \ldots, \ldots, \D H_1^{n - n^* - \rho_0}, \D t\},\\
  & &    
    &\phantom{'}{\scriptstyle\vdots}\\
  &\scriptstyle\Ann(\TangentSpaceAt{\Manifold{L}}{p})
  \cap
  \langle\Ideal{I}^{(\kappa_1 - 1)}, \D t\rangle^{(\infty)}_p&
    &{\scriptstyle=}
    \scriptstyle\SpanWord\{\D H_{\kappa_1-1}^1, \ldots, \D H_{\kappa_1-1}^{\rho_{\kappa_1 - 1}}, \D t\},\\
  &\scriptstyle\Ann(\TangentSpaceAt{\Manifold{L}}{p})
  \cap
  \langle\Ideal{I}^{(\kappa_1)}, \D t\rangle^{(\infty)}_p&
    &{\scriptstyle=}
    \scriptstyle\SpanWord\{ \D t \}.
\end{alignedat}
\end{equation}
None of the previous results guarantee that components of \(H_{k}\) are also components of \(H_{k-1}\) even though we know that
\[
  \GenerateInline{\D H_k^1, \ldots, \D H_k^{\ell_k}}
    =
    \GenerateInline{\Ideal{I}^{(k)}, \D t}^{(\infty)}
    \subseteq
    \GenerateInline{\Ideal{I}^{(k - 1)}, \D t}^{(\infty)}
    =
      \GenerateInline{\D H_{k-1}^1, \ldots, \D H_{k-1}^{\ell_{k-1}}}
\]
The coming lemmas ensure that we can always rewrite~\eqref{eqn:unadapted} as the partially adapted basis,
\begin{equation}
  \label{eqn:adapted}\blue
\begin{alignedat}{2}
  &\scriptstyle\Ann(\TangentSpaceAt{\Manifold{L}}{p})
  \cap
  \langle\Ideal{I}^{(0)}, \D t\rangle^{(\infty)}_p&
    &{\scriptstyle=}
    \scriptstyle\SpanWord\{
        \D H_{\kappa_1 - 1}^1, \ldots, \ldots, \ldots, \D H_0^{n - n^*}, \D t
      \},\\
  &\scriptstyle\Ann(\TangentSpaceAt{\Manifold{L}}{p})
  \cap
  \langle\Ideal{I}^{(1)}, \D t\rangle^{(\infty)}_p&
    &{\scriptstyle=}
    \scriptstyle\SpanWord\{
        \D H_{\kappa_1 - 1}^1, \ldots, \ldots, \D H_1^{n - n^* - \rho_0}, \D t
      \},\\
  && 
    &\phantom{'}{\scriptstyle\vdots}\\
  &\scriptstyle\Ann(\TangentSpaceAt{\Manifold{L}}{p})
  \cap
  \langle\Ideal{I}^{(\kappa_1 - 1)}, \D t\rangle^{(\infty)}_p&
    &{\scriptstyle=}
      \scriptstyle\SpanWord\{
        \D H_{\kappa_1 - 1}^1, \ldots, \D H_{\kappa_1 - 1}^{\rho_{\kappa_1 - 1}}, \D t
      \},\\
  &\scriptstyle\Ann(\TangentSpaceAt{\Manifold{L}}{p})
  \cap
  \langle\Ideal{I}^{(\kappa_1)}, \D t\rangle^{(\infty)}_p&
    &{\scriptstyle=}
      \scriptstyle\SpanWord\{\D t\}.
\end{alignedat}
\end{equation}
Pay close attention to the subtle difference between~\eqref{eqn:adapted} and~\eqref{eqn:unadapted}:
if a differential appears as a generator in one ideal, it appears as a generator in all the preceding ideals of the derived flag.
We would like the components of our smooth maps to satisfy this property.
To do this, noting that \(\ell_k \geq \ell_{k+1},\) define the projection \(P_{k}: \Real^{\ell_{k+1}}\times \Real^{\ell_{k} - \ell_{k+1}} \to \Real^{\ell_{k+1}},\) \(P(x,y) = x.\)
Then, rewrite the components of \({H}_k\) so that the following diagram commutes.
\begin{center}
  \blue
  \begin{tikzcd}[column sep=large]
      \OpenSet{V} \subseteq \Manifold{M} \arrow[r, "{\Identity}"] \arrow[dd, "{{H}_{0}}"]
        &
          \OpenSet{V} \subseteq \Manifold{M} \arrow[r, "{\Identity}"] \arrow[dd, "{{H}_{1}}"]
        &
          \cdots \arrow[r, "{\Identity}"]
        &
          \OpenSet{V} \subseteq \Manifold{M} \arrow[r, "{\Identity}"] \arrow[dd, "{{H}_{n - n^* - 1}}"]
        &
          \OpenSet{V} \subseteq \Manifold{M} \arrow[dd, "{H_{n - n^*}}"]\\
      & & & & \\
      \Real^{\ell_0} \arrow[r, "{P_0}"]
        &
          \Real^{\ell_1} \arrow[r, "{P_1}"]
        &
          \cdots \arrow[r, "{P_{n-n^*-2}}"]
        &
          \Real^{\ell_{n - n^* - 1}} \arrow[r, "{P_{n-n^*-1}}"]
        &
          \Real^{\ell_{n - n^*}}
  \end{tikzcd}
\end{center}
Maps \(H_k\) that make this diagram commute have components that are subsumed in the ``larger'' maps \(H_{k-1},\) \(\ldots,\) \(H_0.\)
This can be done, up to a reordering in the projection, to preserve the fact that the leading components of \(H_k\) vanish on \(\Manifold{L}.\)
To prove that such a construction is possible, we need only prove that a smaller adaptation is possible.
\begin{proposition}
  \label{prop:adaptF}
  \blue
  Let \(k \geq 0,\) \(\ell_k > \ell_{k+1} > 0,\) let \(P_k: \Real^{\ell_{k+1}} \times \Real^{\ell_{k}-\ell_{k+1}} \to \Real^{\ell_{k+1}},\) \(P_k(x,y) = x,\) and let \(\OpenSet{U}\) be an open set containing \(p_0.\)
  If \(H_{k+1}: \OpenSet{U} \to \Real^{\ell_{k+1}}\) and \(H_{k}: \OpenSet{U} \to \Real^{\ell_{k}}\) are smooth maps satisfying the characteristic property~\eqref{lemma:2:a}, then on a possibly smaller open set \(\OpenSet{V}\) containing \(p_0,\) there exists a smooth map \(\tilde{H}_{k}: \OpenSet{V} \to \Real^{\ell_{k}}\) that makes the diagram,
  \begin{center}
    \begin{tikzcd}
      \OpenSet{V} \subseteq \Manifold{M}
      \arrow{d}[left]{\tilde{H}_k}
      \arrow{rrd}[above right]{H_{k+1}}
        &
        &
      \\
      \Real^{\ell_k}
      \arrow[rr, "P_k"]
        &
        &
          \Real^{\ell_{k+1}}
    \end{tikzcd}
  \end{center}
  commute and 
  \(
    \GenerateInline{\D \tilde{H}_k^1, \ldots, \D \tilde{H}_k^{\ell_k}}
  \)
  \(
    =
  \)
  \(
    \GenerateInline{\D H_k^1, \ldots, \D H_k^{\ell_k}}
  \)
  \(
    =
  \)
  \(
    \GenerateInline{\Ideal{I}^{(k)}, \D t}^{(\infty)}.
  \)
\end{proposition}
\begin{proof}
  Write \(H_{k+1}\) \(=\) \((H_{k+1}^1,\) \(\ldots,\) \(H_{k+1}^{\ell_{k+1}}).\)
  Since \(\langle \Ideal{I}^{(k+1)}, \D t \rangle^{(\infty)}\) is contained in \(\langle \Ideal{I}^{(k)}, \D t \rangle^{(\infty)}\) we know that
  \[
    \GenerateInline{\D H_{k+1}^1, \ldots, \D H_{k+1}^{\ell_{k+1}}}
    \subseteq
    \GenerateInline{\D H_k^1, \ldots, \D H_k^{\ell_k}}.
  \]
  It follows that we can pick the \(\ell_{k+1} - \ell_{k}\) components of \(H_{k}\) that are differentially independent from the components \(H_{k+1}^i\) at \(p_0.\)
  Take these differentially independent components of \(H_{k+1}\) to be the last \(\ell_{k+1} - \ell_k\) components without loss of generality.
  Define
  \[
    \tilde{H}_k
      \defineas
        (
          H_{k+1}^1, \ldots, H_{k+1}^{\ell_{k+1}},
          H_{k}^{\ell_{k+1} - \ell_k + 1}, \ldots, H_{k}^{\ell_{k}}
        ),
  \]
  and observe that, on a sufficiently small open set \(\OpenSet{V}\) containing \(p_0,\) this map will satisfy
  \[
    \GenerateInline{\D \tilde{H}_k^1, \ldots, \D \tilde{H}_k^{\ell_k}}
    =
    \GenerateInline{\D H_k^1, \ldots, \D H_k^{\ell_k}}
  \]
  and \(P_k \circ \tilde{H}_k = H_{k+1}.\)
\end{proof}
Proposition~\ref{prop:adaptF} implies that, for every \(0 \leq k \leq n-n^*-1,\) there exists \(H_k, H_{k+1}\) so that \(H_{k+1} = P_k \circ H_k.\)
Geometrically, the level sets of the components of \(H_{k+1}\) are subsumed by the level sets of the components of \(H_k.\)
This is depicted in Figure~\ref{fig:AdaptToF}.
Together with Lemma \ref{lem:adaptToL}, we can find a sequence of maps whose leading components vanish on \(\Manifold{L}\) while making the aforementioned diagram commute (up to a reordering in the projections).
The next corollary states this fact.
\begin{corollary}
  \label{cor:adaptedbasis}
  If~\eqref{controllability} holds and there exists an open set \(\OpenSet{U}\subseteq \Manifold{M}\) containing \(p_0\) where \eqref{involutivity} and~\eqref{constantdim} hold then there exists a possibly smaller open set \(\OpenSet{V} \subseteq \Manifold{U}\) containing \(p_0\) and a sequence of smooth maps \(H_0,\) \(\ldots,\) \(H_{n - n^*}\) so that:
  \begin{enumerate}[label={(\arabic*)}]
    \item{
      each map \(H_k\) satisfies the characteristic property~\eqref{lemma:2:a},
    }
    \item{
      for all \(0 \leq k \leq n-n^*-1,\) \(H_{k+1} = P_k \circ H_k\) where \(P_k: \Real^{\ell_{k+1}} \times \Real^{\ell_{k}-\ell_{k+1}} \to \Real^{\ell_{k+1}}\) is a projection onto the leading \(\ell_{k+1}\) components of \(\Real^{\ell_{k}},\) and
    }
    \item{
      for all \(0 \leq k \leq n - n^*\) the leading \(1 + \sum_{i = k}^{n - n^* - 1} \rho_{i}\) components of \(H_k\) vanish on \(\Manifold{L}.\) 
    }
  \end{enumerate}
\end{corollary}
Corollary~\ref{cor:adaptedbasis} encodes subprocedures~\ref{adapt:a} and~\ref{adapt:b} as presented at the start of this section.
It assures us that there exists a set of generators which ``drop off'' on computing the derived flag while explicitly expressing the components with differentials that annihilate tangent vectors to \(\Manifold{L}.\)
Note, however, that we still do not know what the transverse output is.
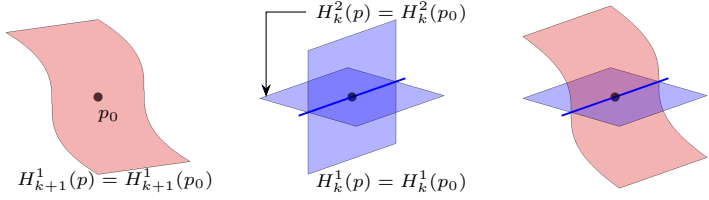
\begin{figure}[h]
  \centering
  \begin{subfigure}[t]{0.25\textwidth}
    \centering
    \begin{tikzpicture}[x=1in, y=1in]
      \node[draw=none, fill=none, anchor={south west}] at (0,0) {%
        \includegraphics[scale=0.5]{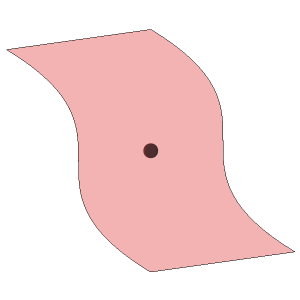}%
      };

      \node[at={(0.5, 0.45)}, anchor={west}, font={\scriptsize}] (lblp0) {%
        \(p_0\)%
      };
      
      \node[at={(1.2, 0)}, anchor={south east}, font={\scriptsize}] (lblF1) {%
      \(H_{k+1}^1(p) = H_{k+1}^1(p_0)\)%
      };
    \end{tikzpicture}
  \end{subfigure}
  \begin{subfigure}[t]{0.25\textwidth}
    \centering
    \begin{tikzpicture}[x=1in, y=1in]
      \node[draw=none, fill=none, anchor={south west}] at (0,0) {%
        \includegraphics[scale=0.5]{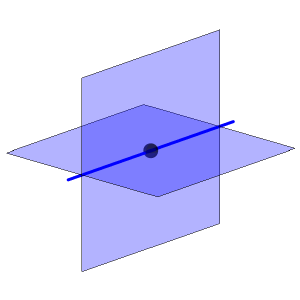}%
      };

      \node[at={(1.2, 0)}, anchor={south east}, font={\scriptsize}] {%
      \(H_{k}^1(p) = H_{k}^1(p_0)\)%
      };

      \node[at={(1.2, 1.1)}, anchor={north east}, font={\scriptsize}] (lblF2) {%
      \(H_{k}^2(p) = H_{k}^2(p_0)\)%
      };
      \draw[-{Stealth}]
        (lblF2.west) -| (0.1, 0.55);

    \end{tikzpicture}
  \end{subfigure}
  \begin{subfigure}[t]{0.25\textwidth}
    \centering
    \begin{tikzpicture}[x=1in, y=1in]
      \node[draw=none, fill=none, anchor={south west}] at (0,0) {%
        \includegraphics[scale=0.5]{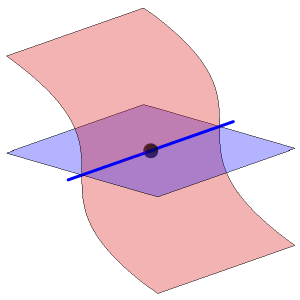}%
      };
    \end{tikzpicture}
  \end{subfigure}
  \caption{%
    A depiction of Proposition~\ref{prop:adaptF}.
    The level sets of components of \(H_k\) subsume those of \(H_{k+1}\) while preserving the image of their differentials.
  }
  \label{fig:AdaptToF}
\end{figure}

\section{The Proposed Algorithm}
\label{sec:proposedalgorithm}

\begingroup\blue
It was purported that the proof that the conditions~\eqref{controllability},~\eqref{involutivity} and~\eqref{constantdim} are sufficient for transverse feedback linearization employs the proposed algorithm.
This is precisely what we aim to show in this section.

\subsection{The Geometry of the Algorithm}
\label{sec:geometryofalg}
The proposed algorithm produces a flag of (locally) closed, embedded submanifolds,
\[
  \Real^n
    \supseteq 
      \pi(\OpenSet{U})
    \eqqcolon
      \Manifold{Z}^{(n-n^*+1)}
    \supseteq
      \Manifold{Z}^{(n-n^*)}
    \supseteq
      \cdots
    \supseteq
      \Manifold{Z}^{(2)}
    \supseteq
      \Manifold{Z}^{(1)}
    =
      \pi(\OpenSet{U}) \cap \Manifold{N}.
\]
Each manifold \(\Manifold{Z}^{(i)}\) is constructed as the local zero dynamics manifold containing \(x_0\) of an incomplete transverse output with respect to \(\Manifold{N}.\)
The scalar outputs used to construct \(\Manifold{Z}^{(i)}\) are all the components of the transverse output for \(\Manifold{N}\) which have relative degree greater than, or equal to, \(i.\)

As a result, one can take the following perspective on the algorithm.
Informally, the algorithm transverse feedback linearizes the system dynamics with respect to \(\Manifold{N}\) by finding those scalar outputs that are transverse to \(\Manifold{N}\) and have the \emph{largest possible} relative degree \(\kappa_1.\)
The algorithm proceeds by finding outputs that yield a lower relative degree, are transverse to \(\Manifold{N}\) but \emph{not} transverse to the zero dynamics manifold \(\Manifold{Z}^{(\kappa_1)}\) induced by the already known outputs.
The new outputs can then be combined with the known outputs to yield a relative degree with a smaller zero dynamics manifold.
The process repeats until the zero dynamics manifold agrees with \(\Manifold{N}\) locally.
To simplify the discussion involving the outputs and their associated zero dynamics manifolds, we define a special class of controlled-invariant set.
\begin{definition}
  \label{def:RegularZeroDynamics}
  A closed, embedded submanifold \(\Manifold{Z} \subseteq \Real^n\) is said to be a \textbf{regular zero dynamics manifold of type }\(\bm{(\ell, \kappa)}\)\textbf{ (at }\(\bm{x_0}\)\textbf{)} for~\eqref{eqn:system}, where \(1 \leq \ell \leq n\) and \(\kappa\in\Naturals^\ell,\) if there exists an open set \(\OpenSet{U} \subseteq \Real^n\) containing \(x_0\) and a smooth function \(h : \OpenSet{U} \to \Real^\ell\) so that
  \begin{enumerate}[label={(\arabic*)}]
    \item{the system~\eqref{eqn:system} with output \(h\) yields a vector relative degree \(\kappa = (\kappa_1, \ldots, \kappa_\ell)\) at \(x_0,\) and}
    \item{the zero dynamics manifold for \(h\) coincides with \(\OpenSet{U} \cap \Manifold{Z}.\)}
  \end{enumerate}
\end{definition}
Although not explicit, Definition~\ref{def:RegularZeroDynamics} implies that a regular zero dynamics manifold \(\Manifold{Z}\) of type \((\ell, \kappa)\) has dimension \(n - \sum_{i = 1}^\ell \kappa_i.\)
It is also clear from Definition~\ref{def:RegularZeroDynamics} that regular zero dynamics manifolds are controlled-invariant sets;
the converse is clearly not true.
Since all of the \(\Manifold{Z}^{(i)}\) in the proposed algorithm's flag are constructed as the zero dynamics manifold associated to some output for~\eqref{eqn:system}, the proposed algorithm produces a flag of regular zero dynamics manifolds containing \(\Manifold{N}.\)
We can restate Theorem~\ref{thm:LTFL-Nielsen-3.1} in the language of Definition~\ref{def:RegularZeroDynamics}.
\begin{theorem}
  \label{thm:LTFL-TowerControlled}
  The local transverse feedback linearization problem is solvable at \(x_0\) if, and only if, there exists constants \(\rho_0 \in \Naturals\) and \(\kappa = (\kappa_1, \ldots, \kappa_{\rho_0}) \in \Naturals^{\rho_0}\) so that \(\Manifold{N}\) is a regular zero dynamics manifold of type \((\rho_0, \kappa)\) at \(x_0.\)
\end{theorem}
Theorem~\ref{thm:LTFL-TowerControlled} implies that the proposed algorithm produces a descending flag of manifolds that are transverse feedback linearizable at \(x_0.\)
% This formulation gives geometric intuition behind the algorithm for transverse feedback linearization, but also to algorithms for feedback linearization methods more broadly.
% The GS algorithm proposed in~\cite{Gardner1992} and Blended algorithm proposed in~\cite{Mehra2014} both involve adapting generators starting with the smallest (last) ideal in the derived flag~\eqref{eqn:flag} and working backwards to the system ideal \(\Ideal{I}^{(0)}.\)
% This can be viewed as constructing a sequence of components for the feedback linearizing output starting with those components that yield the highest relative degree to those that yield the lowest relative degree.
% Implicitly, this produces a descending flag of regular zero dynamics manifolds that terminates at the one-point set \(\{x_0\}.\)
% Partial feedback linearization, which is concerned with finding the largest feedback linearizable subsystem, could be cast in the same light by instead asking that the last regular zero dynamics manifold in the flag be as small (in dimension) as possible.

\subsection{The Proof}
\label{sec:mainproof}
We are ready to address the more interesting direction of the proof for Theorem~\ref{thm:tfl}.
This direction of the proof demonstrates how to sequentially construct the transverse output that is used to perform transverse feedback linearization.
\endgroup
\begin{proof}[Proof of Theorem~\ref{thm:tfl} (Sufficiency)]
  Suppose~\eqref{controllability} holds at \(p_0\) and that there exists an open set \(\OpenSet{U} \subseteq \Manifold{M}\) of \(p_0\) on which conditions~\eqref{involutivity} and~\eqref{constantdim} hold.
  At the start of the algorithm we set \(\Manifold{Z}^{(n-n^*+1)} \defineas \pi(\OpenSet{U})\) since there cannot be a scalar transverse output for \(\Manifold{N}\) with relative degree greater than \(n-n^*.\)

  By Corollary~\ref{cor:adaptedbasis} there exists a sequence of smooth maps \(H_0,\) \(\ldots,\) \(H_{n - n^*}\) defined on an open set \(\OpenSet{V} \subseteq \OpenSet{U}\) containing \(p_0\) that satisfy the characteristic property~\eqref{lemma:2:a}, i.e., for any fixed \(0 \leq k \leq n - n^*,\)
  \[
    \GenerateInline{ \D H_k^1, \ldots, \D H_k^{\ell_k} }
    =
    \GenerateInline{ \Ideal{I}^{(k)}, \D t }^{(\infty)}.
  \]
  Furthermore, the leading components of these maps vanish on \(\Manifold{L}\) so we may write, for any \(0 \leq k \leq n - n^*,\)
  \[
    H_k
      =
      (
        \underbrace{
          H_{k}^1, \ldots, t
        }_{%
          \text{vanish on}\:\Manifold{L}%
        },
        \underbrace{
          \ldots, \ldots, \ldots, H_k^{\ell_k}
        }_{%
          \text{not constant on}\:\Manifold{L}%
        }
      ).
  \]
  Now recall the discussion surrounding~\eqref{eqn:rho-to-nstar}.
  We already know that the \(t\) component of \(H_{n-n^*}\) is the only component that vanishes on \(\Manifold{L}.\)
  However, since \(\rho_{\kappa_1}\) \(=\) \(\cdots\) \(=\) \(\rho_{n-n^*-1} = 0,\) we can also say the same thing about the map \(F_{\kappa_1}.\)
  That is,
  \begin{equation}
    \label{eqn:main:FKappa1}
    H_{\kappa_1}
    =
    (
      t,
      \underbrace{
        H_{\kappa_1}^2, \ldots, \ldots, H_{\kappa_1}^{\ell_{\kappa_1}}
      }_{%
        \text{not constant on}\:\Manifold{L}%
      }
    ),
  \end{equation}
  As a result, we start our algorithm at index \(\kappa_1.\)

  Set \(\Manifold{Z}^{(\kappa_1 + 1)}\) \(\defineas\) \(\Manifold{Z}^{(n-n^*+1)}\) \(=\) \(\pi(\OpenSet{U}),\) and consider the map \(H_{\kappa_1-1}.\)
  By definition, \(\rho_{\kappa_1 - 1} > 0.\)
  Therefore, we can write
  \[
    H_{\kappa_1 - 1}
    =
    (
      \underbrace{
        H_{\kappa_1 - 1}^1,
        \ldots,
        H_{\kappa_1 - 1}^{\rho_{\kappa_1 - 1}},
        t
      }_{\text{vanish on}\:\Manifold{L}},
      \underbrace{
        \ldots, \ldots, \ldots, H_{\kappa_1-1}^{\ell_{\kappa_1-1}}
      }_{%
        \text{not constant on}\:\Manifold{L}%
      }
    ).
  \]
  Take the \(\rho_{\kappa_1 - 1}\) smooth functions \(H_{\kappa_1 - 1}^1,\) \(\ldots,\) \(H_{\kappa_1 - 1}^{\rho_{\kappa_1 - 1}},\) and define the candidate output
  \[
    h \defineas (
      H_{\kappa_1 - 1}^1, \ldots, H_{\kappa_1 - 1}^{\rho_{\kappa_1 - 1}}
    ).
  \]
  Now we show that the system~\eqref{eqn:systemOnM} with output \(h\) yields a vector relative degree of \((\kappa_1,\) \(\ldots,\) \(\kappa_1)\) at \(p_0.\)
  Clearly \(\D h^i \in \langle \Ideal{I}^{(\kappa_1 - 1)}, \D t\rangle^{(\infty)}\) by the characteristic property~\eqref{lemma:2:a}.
  Therefore
  \begin{equation}
    \label{main:base:1}
    \GenerateInline{
      \D h^1,\ldots, \D h^{\rho_{\kappa_1 - 1}}
    }
    \subseteq
    \GenerateInline{
      \Ideal{I}^{(\kappa_1 - 1)}, \D t
    }^{(\infty)}.
  \end{equation}
  Observe that \(\D h^i \in \Ann(\TangentSpaceOf{\Manifold{L}}).\)
  As a result,
  \[
    \SpanInline{\Real}{
      \D h^1_{p_0}, \ldots, \D h^{\rho_{\kappa_1 - 1}}_{p_0}
    }
      \subseteq
        \Ann(\TangentSpaceAt{\Manifold{L}}{p_0}).
  \]
  Using~\eqref{eqn:main:FKappa1} and the characteristic property~\eqref{lemma:2:a},
  \[
    \SpanInline{\Real}{
      \D h^1_{p_0},\ldots, \D h^{\rho_{\kappa_1 - 1}}_{p_0}
    }
    \cap
    \GenerateInline{
      \Ideal{I}^{(\kappa_1)},
      \D t
    }_{p_0}^{(\infty)}
    \subseteq
    \SpanInline{\Real}{\D t_{p_0}}.
  \]
  We already know that the \(h^i\) are smooth functions of the state, so we may conclude
  \[
    \SpanInline{\Real}{
      \D h^1_{p_0},\ldots, \D h^{\rho_{\kappa_1 - 1}}_{p_0}
    }
    \cap
    \GenerateInline{
      \Ideal{I}^{(\kappa_1)},
      \D t
    }_{p_0}^{(\infty)}
    = 
      \{0\}.
  \]
  Finally use Proposition~\ref{prop:invequals} to deduce
  \begin{equation}
    \label{main:base:2}
    \SpanInline{\Real}{
      \D h^1_{p_0},\ldots, \D h^{\rho_{\kappa_1 - 1}}_{p_0}
    }
    \cap
    \SpanInline{\Real}{
      \Codist{I}^{(\kappa_1)}_{p_0},
      \D t_{p_0}
    }
    = 
      \{0\}.
  \end{equation}
  The expressions~\eqref{main:base:1} and~\eqref{main:base:2} are the conditions for Proposition~\ref{prop:relativedegree}.
  Conclude that \(h\) yields a uniform vector relative degree of \((\kappa_1,\) \(\ldots,\) \(\kappa_1)\) at \(p_0.\)
  Define the local regular zero dynamics manifold, possibly shrinking \(\OpenSet{U}\) if necessary,
  \[ 
    \Manifold{Z}^{(\kappa_1)}
      \defineas 
        \{
          x \in \pi(\OpenSet{U})
          \colon
          h(\iota(x)) = \cdots = \Lie_f^{\kappa_1 - 1} h(\iota(x)) = 0
        \},
  \]
  of type \((\rho_{\kappa_1 - 1}, (\kappa_1,\ldots,\kappa_{\rho_{\kappa_1 - 1}}))\) at \(x_0 = \pi(p_0).\)
  Fix \(0 \leq k < \kappa_1 - 1.\)
  \begingroup\blue
  It is easily demonstrated that for any smooth function of the states \(h\) where \(\D h \in \GenerateInline{\Ideal{I}^{(\kappa_1)}, \D t}\) that \(\Lie_f \D h \in \GenerateInline{\Ideal{I}^{(\kappa_1 - 1)}, \D t}\);
  it follows by an application of~\cite[Lemma 2.2.5]{DSouza2022} and~\cite[Lemma 2.2.6]{DSouza2022}.
  This, in combination with~\eqref{main:base:1} implies that
  \endgroup
  \[
    \D h^i,
    \ldots,
    \Lie_f^{\kappa_1 - k - 1} \D h^i
    \in
    \GenerateInline{\Ideal{I}^{(k)}, \D t}^{(\infty)}
    =
    \GenerateInline{ \D F_k^1, \ldots, \D F_k^{\ell_k} },
    \qquad
    1 \leq i \leq \rho_{\kappa_1 - 1}.
  \]
  As a result, we can, without loss of generality, rewrite \(H_k\) to take the form
  \[
    H_k
      =
        (
          \underbrace{
            \overbrace{
              \ldots, \ldots, \vphantom{\Lie_f^{\kappa_1 - 1 - k} h}\ldots, \ldots
            }^{\text{other components}},
            \overbrace{
              h, \ldots, \Lie_f^{\kappa_1 - k - 1} h
            }^{\text{vanish on}\:\Manifold{Z}^{(\kappa_1)}},
            \; t
          }_{%
            \text{vanish on}\:\Manifold{L}%
          },
          \underbrace{
            \ldots,\;\ldots,\;\ldots,\;\ldots
          }_{%
            \text{not constant on}\:\Manifold{L}%
          }
        ),
  \]
  This process adapts all maps \(H_k\) for \(k < \kappa_1 - 1\) so that they explicitly include \(h\) and its Lie derivatives along \(f\) in their components.
  Perform this operation for each \(k\) while ensuring the components of \(H_k\) are subsumed by the components of \(H_{k-1},\) i.e., \(H_k = P_{k-1} \circ H_{k-1},\) up to a reordering.
  We say that the sequence \(H_0, \ldots, H_{\kappa_1}\) is adapted to \(\Manifold{L}\) subordinate to the regular zero dynamics manifold \(\Manifold{Z}^{(\kappa_1)}.\)
  Observe that \(\Manifold{Z}^{(\kappa_1)} \supseteq \Manifold{N}\) since \(\left.h\right|_{\Manifold{N}} = 0.\)
  It is trivially the case that \(\Manifold{Z}^{(\kappa_1)} \subset \Manifold{Z}^{(\kappa_1 + 1)} = \pi(\OpenSet{U}).\)
  The final fact that we simply state is that the only differentials that vanish on \(\Manifold{L}\) at index \(\kappa_1 - 1\) are those that are linearly dependent on the differentials of \(h.\)
  That is,
  \[
    \Ann(\TangentSpaceAt{\Manifold{L}}{p_0})
    \cap
    \SpanInline{\Real}{
      \D H_{\kappa_1 - 1}^1
      \ldots,
      \D H_{\kappa_1 - 1}^{\ell_{\kappa_1 - 1}}
    }
    \subseteq
    \Ann(\TangentSpaceAt{\Manifold{L}^{(\kappa_1)}}{p_0}),
  \]
  where \(\Manifold{L}^{(\kappa_1)}\) is the lift of \(\Manifold{Z}^{(\kappa_1)}.\)
  This completes the base case.
  Suppose, by way of induction, that, for some \(2 \leq k \leq \kappa_1,\)
  \begin{enumerate}[label={(H.\arabic*)}]
    \item{
      \label{main:induction:1}
      \(\Manifold{Z}^{(k)}\) is a regular zero dynamics manifold of type \((\rho_{k-1}, (\kappa_1, \ldots, \kappa_{\rho_{k-1}}))\) at \(x_0\) satisfying
      \(
        \Manifold{Z}^{(k + 1)} \supseteq \Manifold{Z}^{(k)} \supseteq \Manifold{N},
      \)
    }
    \item{
      \label{main:induction:2}
      there exists a smooth function \(h: \OpenSet{U} \to \Real^{\rho_{k-1}}\) so that system~\eqref{eqn:system} with output \(h\) yields a vector relative degree \((\kappa_1,\) \(\ldots,\) \(\kappa_{\rho_{k-1}})\) at \(x_0\) and the zero dynamics coincide locally with \(\Manifold{Z}^{(k)},\)
    }
    \item{
      \label{main:induction:3}
      all maps \(H_0, \ldots, H_{n - n^*}\) are adapted to \(\Manifold{L}\) subordinate to the regular zero dynamics manifold \(\Manifold{Z}^{(k)}\) using output \(h\) and,
    }
    \item{
      \label{main:induction:4}
      denoting \(\Manifold{L}^{(k)}\) as the lift of \(\Manifold{Z}^{(k)}\) we have that
      \begin{equation}
        \label{eqn:main:inductive:2}
        \Ann(\TangentSpaceAt{\Manifold{L}}{p_0})
        \cap
        \SpanInline{\Real}{
          \D H_{k-1}^1,
          \ldots,
          \D H_{k-1}^{\ell_{k-1}}
        }
        \subseteq
        \Ann(\TangentSpaceAt{\Manifold{L}^{(k)}}{p_0}),
      \end{equation}
    }
  \end{enumerate}
  The goal of this induction is to construct new regular zero dynamics manifold \(\Manifold{Z}^{(k-1)}\) that satisfies
  \begin{enumerate}[label={(C.\arabic*)}]
    \item{
      \label{main:induction:1c}
      \(\Manifold{Z}^{(k-1)}\) is a regular zero dynamics manifold of type \((\rho_{k-2}, (\kappa_1, \ldots, \kappa_{\rho_{k-2}}))\) at \(x_0\) satisfying
      \(
        \Manifold{Z}^{(k)} \supseteq \Manifold{Z}^{(k - 1)} \supseteq \Manifold{N},
      \)
    }
    \item{
      \label{main:induction:2c}
      there exists a smooth function \(h': \OpenSet{U} \to \Real^{\rho_{k-2}}\) so that system~\eqref{eqn:system} with output \(h'\) yields a vector relative degree \((\kappa_1,\) \(\ldots,\) \(\kappa_{\rho_{k-2}})\) at \(x_0\) and the zero dynamics coincide locally with \(\Manifold{Z}^{(k-1)},\)
    }
    \item{
      \label{main:induction:3c}
      all maps \(H_0, \ldots, H_{n - n^*}\) are adapted to \(\Manifold{L}\) subordinate to the regular zero dynamics manifold \(\Manifold{Z}^{(k-1)}\) using output \(h'\) and,
    }
    \item{
      \label{main:induction:4c}
      denoting \(\Manifold{L}^{(k-1)}\) as the lift of \(\Manifold{Z}^{(k-1)}\) we have that
      \begin{equation}
        \label{eqn:main:induction:4c}
        \Ann(\TangentSpaceAt{\Manifold{L}}{p_0})
        \cap
        \SpanInline{\Real}{
          \D H_{k-2}^1,
          \ldots,
          \D H_{k-2}^{\ell_{k-2}}
        }
        \subseteq
        \Ann(\TangentSpaceAt{\Manifold{L}^{(k-1)}}{p_0}),
      \end{equation}
    }
  \end{enumerate}
  Consider the map \(H_{k-2}\) whose component differentials generate the differential ideal
  \(
    \langle \Ideal{I}^{(k - 2)}, \D t\rangle^{(\infty)}.
  \)
  By~\ref{main:induction:3} of the inductive hypothesis, \(H_{k-2}\) takes the form
  \[
    H_{k-2}
      =
      (
        \underbrace{
          \overbrace{
            H_{k-2}^1, \ldots, \ldots, H_{k-2}^{\rho_{k-2} - \rho_{k-1}}
          }^{\text{other components}},
          \overbrace{
            h^1, \ldots, \Lie_f^{\kappa_{\rho_{k-1}} - k + 1} h^{\rho_{k-1}}
          }^{\text{vanish on}\:\Manifold{Z}^{(k)}},
          \; t
        }_{%
          \text{vanish on}\:\Manifold{L}%
        },
        \underbrace{
          \ldots,\;\ldots,\;\ldots
        }_{%
          \text{not constant on}\:\Manifold{L}%
        }
      ).
  \]
  There are now two cases.
  If \(\rho_{k-2} = \rho_{k-1},\) then the number of ``other components'' that vanish on \(\Manifold{L}\) is zero.
  This is because, due to the vector relative degree of \(h,\) the \(\rho_{k-1}\) new components 
  \(
    \Lie_f^{\kappa_{1} - k + 1} h^1, \ldots, \Lie_f^{\kappa_{\rho_{k-1}} - k + 1} h^{\rho_{k-1}},
  \)
  appear in \(F_{k-2}.\)
  In this case, set \(\Manifold{Z}^{(k-1)} = \Manifold{Z}^{(k)}\) which remains a regular zero dynamics manifold of type \((\rho_{k-2}, (\kappa_1, \ldots, \kappa_{\rho_{k-2}}))\) establishing~\ref{main:induction:1c}.
  The rest of the inductive properties~\ref{main:induction:2c}--\ref{main:induction:4c} follow directly from~\ref{main:induction:2}--\ref{main:induction:4} by not changing the output \(h' \defineas h.\)

  Alternatively, \(\rho_{k-2} > \rho_{k-1}.\)
  In this case, there exists precisely \(\mu \defineas \rho_{k-2} - \rho_{k-1}\) ``new'' components whose differentials annihilate tangent vectors to \(\Manifold{L}\):
  these are the first components that are differentially independent of the Lie derivatives of \(h\) yet vanish on \(\Manifold{L}.\)
  Take these new component functions, up to a reordering, to be the leading components
  \(
    H_{k-2}^{1}, \ldots, H_{k-2}^{\mu},
  \)
  and define the output
  \[
    q \defineas (
      H_{k-2}^{1}, \ldots, H_{k-2}^{\mu}
    ).
  \]
  We now show that the system~\eqref{eqn:systemOnM} with candidate output \(h' \defineas (h, q)\) yields a well-defined vector relative degree of \((\kappa_1,\) \(\ldots,\) \(\kappa_{\rho_{k-2}})\) at \(p_0.\)
  Since the \(q^i\) are component functions of \(H_{k-2}\) we have by the characteristic property~\eqref{lemma:2:a} that
  \[
    \GenerateInline{
      \D q^1,
      \ldots,
      \D q^{\mu}
    }
    \subseteq
    \langle \Ideal{I}^{(k - 2)}, \D t\rangle^{(\infty)}.
  \]
  We also know from~\ref{main:induction:2} of the inductive hypothesis that \(h^i\) yields a relative degree of \(\kappa_i\) so the \(j\)th Lie derivative of \(h\) along \(f\) yields a relative degree as well of \(\kappa_i - j,\) for \(0 \leq j \leq \kappa_i - 1.\)
  Invoke Proposition~\ref{prop:relativedegree} to find
  \[
    \GenerateInline{
      \Lie_f^{\kappa_1 - k + 1} \D h^1,
      \ldots,
      \Lie_f^{\kappa_{\rho_{k-1}} - k + 1} \D h^{\rho_{k-1}}
    }
    \subseteq
    \GenerateInline{\Ideal{I}^{(k - 2)}, \D t}^{(\infty)}.
  \]
  Combine these data to find
  \begin{equation}
    \label{eqn:main:3}
    \GenerateInline{
      \D q^1,
      \ldots,
      \D q^{\mu},
      \Lie_f^{\kappa_1 - k + 1} \D h^1,
      \ldots,
      \Lie_f^{\kappa_{\rho_{k-1}} - k + 1} \D h^{\rho_{k-1}}
    }
    \subseteq
    \GenerateInline{\Ideal{I}^{(k - 2)}, \D t}^{(\infty)}.
  \end{equation}
  Putting that aside, invoke Proposition~\ref{prop:relativedegree} once again to find
  \[
    \SpanInline{\Real}{
      \Lie_f^{\kappa_1 - k + 1} \D h^1_{p_0},
      \ldots,
      \Lie_f^{\kappa_{\rho_{k-1}} - k + 1} \D h^{\rho_{k-1}}_{p_0}
    }
    \cap
    \SpanInline{\Real}{
      \Codist{I}^{(k - 1)}_{p_0},
      \D t_{p_0}
    }
    =
      \{ 0 \}.
  \]
  By~\ref{main:induction:4} of the inductive hypothesis,
  \[
    \SpanInline{\Real}{
      \D q^1_{p_0}
      \ldots,
      \D q^{\mu}_{p_0}
    }
    \cap
    \GenerateInline{
      \D H_{k-1}^1,\ldots,\D H_{k-1}^{\ell_{k-1}}
    }
    =
      \{ 0 \},
  \]
  since the \(\D q^i \in \Ann(\TangentSpaceAt{\Manifold{L}}{p_0})\) but \(\D q^i \notin \Ann(\TangentSpaceAt{\Manifold{L}^{(k)}}{p_0}).\)
  Combine these data to conjecture that
  \[
    \SpanInline{\Real}{
      \D q^1_{p_0},
      \ldots,
      \D q^{\mu}_{p_0},
      \Lie_f^{\kappa_1 - k + 1} \D h^1_{p_0},
      \ldots,
      \Lie_f^{\kappa_{\rho_{k-1}} - k + 1} \D h^{\rho_{k-1}}_{p_0}
    }
  \]
  has a trivial intersection with
  \(
    \SpanWord\{
      \Codist{I}^{(k-1)}_{p_0},
      \D t_{p_0}
    \}.
  \)
  Suppose, in search of a contradiction, there is a linear combination
  \[
    \textstyle\sum_{i = 1}^\mu a_i\, \D q^i_{p_0}
    +
    \textstyle\sum_{i = 1}^{\rho_{k-1}} b_i\, \Lie_f^{\kappa_{i} - k + 1} \D h^i_{p_0}
    \in
    \SpanInline{\Real}{
      \Codist{I}^{(k - 1)}_{p_0},
      \D t_{p_0}
    }.
  \]
  If there is an \(a_i\neq 0,\) then this form does not live in \(\Ann(\TangentSpaceAt{\Manifold{L}^{(k)}}{p_0})\) which contradicts~\ref{main:induction:4} of the inductive hypothesis.
  Therefore \(a_i = 0\) for all \(1 \leq i \leq \mu.\) 
  We now show that \(b_i = 0\) for all \(1 \leq i \leq \rho_{k-1}.\)
  Suppose
  \[
    \textstyle\sum_{i = 1}^{\rho_{k-1}} b_i\, \Lie_f^{\kappa_{i} - k + 1} \D h^i_{p_0}
    \in
    \SpanInline{\Real}{
      \Codist{I}^{(k-1)}_{p_0},
      \D t_{p_0}
    },
  \]
  Then, by Proposition~\ref{prop:relativedegree}, the system~\eqref{eqn:systemOnM} with output
  \[
    (\Lie_f^{\kappa_1 - k + 1} h^1, \ldots, \Lie_f^{\kappa_{\rho_{k-1}} - k + 1} h^{\rho_{k-1}})
  \]
  does \emph{not} yield a uniform vector relative degree at \(p_0.\)
  This immediately contradicts the well-defined vector relative degree for \(h.\)
  Therefore \(b_i = 0\) for all \(1 \leq i \leq \rho_{k-1}.\)
  As a result, conclude that system~\eqref{eqn:systemOnM} with output \(h' = (h, q)\) yields a vector relative degree at \(p_0.\)
  The vector relative degree must be \((\kappa_1,\) \(\ldots,\) \(\kappa_{\rho_{k-2}}).\)
  This demonstrates~\ref{main:induction:2c}.
  Define the local regular zero dynamics manifold
  \[ 
    \Manifold{Z}^{(k-1)}
      \defineas 
        \{
          x \in \Manifold{Z}^{(k)}
          \colon
          q(\iota(x)) = \cdots = \Lie_f^{\kappa_{\rho_{k-2}} - 1} q(\iota(x)) = 0
        \},
  \]
  of type \((\rho_{k-2}, (\kappa_1,\) \(\ldots,\) \(\kappa_{\rho_{k-2}})).\)
  By construction \(\Manifold{Z}^{(k-1)} \subset \Manifold{Z}^{(k)}\) and, since \(\left.q\right|_\Manifold{N} = 0,\) \(\Manifold{Z}^{(k-1)} \supseteq \Manifold{N}.\)
  This establishes~\ref{main:induction:1c}.
  To establish~\ref{main:induction:3c}, we adapt, exactly as in the base case, the maps \(H_0,\) \(\ldots,\) \(H_{n - n^*}\) to \(\Manifold{L}\) subordinate to \(\Manifold{Z}^{(k-1)}\) so that all the Lie derivatives of \(h' = (h, q)\) appear explicitly.
  It remains to show~\ref{main:induction:4c}.
  First observe that the components of
  \[
    H_{k-2}
      =
      (
        \underbrace{
          \overbrace{
            q^1, \ldots, \ldots, q^{\mu}
          }^{\text{other components}},
          \overbrace{
            h^1, \ldots, \Lie_f^{\kappa_{\rho_{k-1}} - k + 1} h^{\rho_{k-1}}
          }^{\text{vanish on}\:\Manifold{Z}^{(k)}},
          \; t
        }_{%
          \text{vanish on}\:\Manifold{L},\Manifold{Z}^{(k-1)}%
        },
        \underbrace{
          \ldots,\;\ldots,\;\ldots
        }_{%
          \text{not constant on}\:\Manifold{L}%
        }
      ),
  \]
  that vanish on \(\Manifold{L}\) constitute a component of \(h,\) a Lie derivative of \(h,\) or \(q.\)
  It follows that, using the characteristic property~\ref{lemma:2:a},
  \[
    \Ann(\TangentSpaceAt{\Manifold{L}}{p})
    \cap
    \GenerateInline{
      \Ideal{I}^{(k-2)},
      \D t
    }^{(\infty)}
      \subseteq
        \Ann(\TangentSpaceAt{\Manifold{L}^{(k-1)}}{p}).
  \]
  Use~\eqref{involutivity} with Proposition~\ref{prop:invequals} to conclude that~\ref{main:induction:4c} holds.
  This completes the induction.

  The inductive algorithm proceeds until the regular zero dynamics manifold \(\Manifold{Z}^{(1)}\) of type \((\rho_0, (\kappa_1, \ldots, \kappa_{\rho_0}))\) is produced at step \(k = 2.\)
  By Lemma~\ref{lemma:1}, \(\Manifold{Z}^{(1)}\) is an \(n^*\)-dimensional submanifold with codimension \(n - n^*.\)
  It contains the \(n^*\)-dimensional submanifold \(\Manifold{N}\) and so \(\Manifold{Z}^{(1)} = \Manifold{N}\) is a regular zero dynamics manifold of type \((\rho_0,\) \((\kappa_1,\) \(\ldots,\) \(\kappa_{\rho_0})).\)
  Theorem~\ref{thm:LTFL-TowerControlled} implies that \(\Manifold{N}\) is transverse feedback linearizable at \(x_0.\)
  A by-product of this algorithm is that the final output \(h\) is the transverse output.
\end{proof}
\begingroup\blue
\subsection{Simplifying The Algorithm}
\label{sec:thealgorithm}
The algorithm used in the proof of Section~\ref{sec:mainproof} inspires a shortened, but equivalent, algorithm that produces a transverse output under the conditions for TFL --- \eqref{controllability},~\eqref{involutivity},~\eqref{constantdim}.
The procedure is presented in Algorithm~\ref{alg:tfl}.
\begin{algorithm}
\begin{algorithmic}[1]
  \Procedure{TFL Procedure}{$f, g, \Manifold{N}, x_0, u_*$}
    \State{%
      Compute \(\rho_0(p_0), \ldots, \rho_{n - n^*}(p_0)\)%
    }
    \label{alg:tfl:rho}
    \Comment{as in~\eqref{eqn:controlindex}}
    \State{%
      Compute \(\kappa_1(p_0), \ldots, \kappa_{m}(p_0)\)%
    } \Comment{as in~\eqref{eqn:controlindex:k}}
    \State{%
      \(\Manifold{Z}^{(\kappa_1 + 1)} \gets \Real^n\)
    }
    \State{%
      Initialize \(h \gets (\phantom{x^1})\)%
    }
    \For{\(k \gets \kappa_1, \ldots, 1\)}
      \If{\(\rho_{k - 1} = \rho_k\)}
        \State{%
          \(\Manifold{Z}^{(k)} \gets \Manifold{Z}^{(k+1)}\)%
        }
        \State{\textbf{continue}}
      \EndIf
      \State{%
        \(\mu \gets \rho_{k-1} - \rho_k\)
      }
      \State{%
        Construct \(H_{k-1}\) to satisfy~\eqref{lemma:2:a}%
      } \Comment{integrate \(\langle \Ideal{I}^{(k-1)}, \D t\rangle^{(\infty)}\)}
      \label{alg:tfl:integrate}
      \State{%
        Rewrite \(H_{k-1},\) while preserving~\eqref{lemma:2:a}, so that
        \[
          H_{k-1}
            =
              (
                \underbrace{
                  H_{k-1}^1, 
                  \ldots,
                  H_{k-1}^{\mu},
                  h^1,
                  \ldots,
                  \ldots,
                  \Lie_f^{\kappa_{\rho_k} - k} h^{\rho_{k}},
                  t
                }_{\text{vanish on}\:\Manifold{L}},
                \ldots
              ).
        \]%
      }
      \label{alg:tfl:adapt}
      \State{%
        \(h \gets (h, H_{k-1}^1, \ldots, H_{k-1}^{\mu})\)%
      }
      \State{%
        \(\Manifold{Z}^{(k)} \gets \text{zero dynamics of}\:h.\)%
      }
    \EndFor
  \EndProcedure
\end{algorithmic}
\caption{The Transverse Feedback Linearization algorithm.}
\label{alg:tfl}
\end{algorithm}
One difference from the proof is that integration only happens at iterations where \(\rho_{k-1}\) differs from \(\rho_k.\)
These are indices corresponding to the \emph{distinct} transverse controllability indices.
Another difference is the lack of re-adaptation of all the maps \(H_0,\) \(\ldots,\) \(H_{n-n^*}\) throughout the algorithm.
In fact, the vast majority of the maps \(H_k\) are not constructed.
This is \emph{not} an oversight.
The adaptation process is embedded in Line~\ref{alg:tfl:adapt} where \(H_{k-1}\) is adapted to have the known output \(h\) and its Lie derivatives appear explicitly.
This alongside the fact that the algorithm runs from larger to smaller indices ensure an appropriately adapted basis is constructed.
\endgroup

\section{Example: Performing the Algorithm}
\label{sec:example:algorithm}
Having established the algorithm, let us return to the example in Section~\ref{sec:example:conditions}.
We already verified the conditions for TFL hold.
As a result, we can execute Algorithm~\ref{alg:tfl}.
We start by computing the indices \(\rho\) and \(\kappa\) as required by Line~\ref{alg:tfl:rho}.
Using~\eqref{eqn:example:flag}, deduce
\[
  \rho_0(p_0) = 2,\quad
  \rho_1(p_0) = 2,\quad
  \rho_2(p_0) = 1,\quad
  \rho_3(p_0) = 0,
\]
and,
\[
  \kappa_1(p_0) = 3,\quad
  \kappa_2(p_0) = 2.
\]
Let
\(
  \OpenSet{U}
    \defineas
      \{
        (t, u, x) \in \Manifold{M}
        \colon
        x^1 > 0,
        x^3 > -3
      \}
\)
and set \(\Manifold{Z}^{(\kappa_1 + 1)}\) \(=\) \(\Manifold{Z}^{(4)}\) \(\defineas\) \(\pi(\OpenSet{U}).\)
The ideals are simply, finitely, non-degenerately generated on \(\OpenSet{U}\) and it can be verified that the conditions~\eqref{constantdim} and~\eqref{involutivity} hold over \(\OpenSet{U}.\)

The iteration begins at \(k = \kappa_1 = 3.\)
Observe that \(\rho_{2} > \rho_{3}.\)
It follows that we expect to find \(\mu = \rho_{2} - \rho_3 = 1\) new scalar output that will yield a relative degree of \(3\) at \(p_0\) and is constant on \(\Manifold{L}.\)
We proceed by integrating the differential ideal \(\langle \Ideal{I}^{(2)}, \D t\rangle^{(\infty)}\) to find the map \(H_2 = (x^5 + x^7, t).\)
Right away we see that the new component that is constant (i.e. vanishes) on \(\Manifold{L}\) is the first component and so we define our candidate partial transverse output \(h \defineas x^5 + x^7.\)
We then define the local regular zero dynamics manifold \(\Manifold{Z}^{(3)}\) to be the zero dynamics of system~\eqref{eqn:example:1} with output \(h.\)
Explicitly,
\[
  \Manifold{Z}^{(3)}
    \defineas
      \left\{
        x \in \Real^n
        \colon 
        x^5 + x^7
        =
        x^5 + x^6
        =
        -x^3\, x^5 + 2\, x^6 + x^7
        =
        0
      \right\}.
\]

The next iteration of the algorithm looks at index \(k = 2.\)
Here, we again note that \(\rho_{1} > \rho_2\) and \(\mu = 1.\)
We expect to see one new scalar output with relative degree \(2\) at \(p_0\) that is constant on \(\Manifold{L}.\)
Integrate \(\langle \Ideal{I}^{(1)}, \D t\rangle^{(\infty)}\) to find
\[
  H_1 =
  (
    \overbrace{
      x^5 + x^7,
      x^5 + x^6
    }^{h,\,\Lie_f h},
    \frac{1}{2}(x^1)^2 + x^2\, x^7 - 2,
    x^2,
    x^3\, e^{-x^4} - 4,
    t
  ),
\]
where we highlight the fact that, as expected, the known output \(h\) and its Lie derivative \(\Lie_f h\) appear explicitly in the first two components.
Unfortunately, it is not obvious (at first glance) what the new scalar output is since none of the other components are constant on \(\Manifold{L}\) besides the trivial \(t\) component.
However, Lemma~\ref{lem:adaptToL} states that a rewriting for \(H_1\) where four components vanish on \(\Manifold{L}\) is possible.
The simplest strategy is to restrict \(H_1\) to \(\Manifold{L}\) and eliminate the effect of the coordinates algebraically.
{\blue We perform this to find that the algebraic combination \(2\, H_1^3 + (H_1^4)^2 - H_1^5\) of components of \(H_1\) vanish on \(\Manifold{L}.\)}
Using this, rewrite \(H_1\) as
\[
  H_1 =
  (
    \underbrace{
      (x^1)^2 + (x^2)^2 + 2\, x^2\, x^7 - x^3\, e^{-x^4},
      \overbrace{
        x^5 + x^7,
        x^5 + x^6
      }^{h,\,\Lie_f h},
      t
    }_{\text{vanish on}\:\Manifold{L}},
    \underbrace{
      x^2,
      x^3\, e^{-x^4} - 4
    }_{\text{not constant on}\:\Manifold{L}}
  ).
\]
Define the new candidate output \(h \defineas (x^5 + x^7, {H}_1^1)\) and the induced local, regular zero dynamics manifold, shrinking \(\OpenSet{U}\) as necessary,
\[
  \Manifold{Z}^{(2)}
    \defineas
      \{
        x \in \Manifold{Z}^{(1)}
        \colon
        H_1^1(x) = \Lie_f H_1^1(x) = 0
      \}.
\]

The final iteration of the algorithm at index \(k = 1\) is skipped since \(\rho_0 = \rho_1.\)
Set \(\Manifold{Z}^{(1)} \defineas \Manifold{Z}^{(2)}.\)
The algorithm asserts that \(\Manifold{Z}^{(1)} = \Manifold{N}\) locally.
The transverse output is
\[
  h(x) =
    \begin{bmatrix}
      x^5 + x^7\\
      (x^1)^2 + (x^2)^2 + 2\, x^2\, x^7 - x^3\, e^{-x^4}
    \end{bmatrix},
\]
and it is a regular matter to verify that the system~\eqref{eqn:example:1} with output \(h\) yields a vector relative degree of \((3, 2)\) at \(x_0\) while locally vanishing on \(\Manifold{N}.\)

\section{Conclusion}
We presented an algorithm that, subject to dual conditions for transverse feedback linearization, constructs a transverse output that can be used to put a nonlinear control system into TFL normal form.
The algorithm provides a geometric take on an otherwise algebraic process by viewing adaptation as the subsumption of level sets of a sequence of smooth maps.
Our algorithm subsumes the GS and Blended algorithms for exact state-space feedback linearization in the case where \(\Manifold{N} = \{ x_0 \}\) and therefore gives a geometric perspective on their algorithms as the construction of a descending sequence of regular zero dynamics manifolds.
The structure of the presented algorithm suggests an avenue for future research wherein a variation produces a descending sequence of regular zero dynamics manifolds that terminate on the zero dynamics manifold of the largest feedback linearizable subsystem.

\bibliography{mimo-tfl}

\begin{thebibliography}{10}

\bibitem{Ames2019}
{\sc A.~D. Ames, S.~Coogan, M.~Egerstedt, G.~Notomista, K.~Sreenath, and
  P.~Tabuada}, {\em Control barrier functions: Theory and applications}, in
  2019 18th European Control Conference ({ECC}), {IEEE}, jun 2019,
  \url{https://doi.org/10.23919/ecc.2019.8796030}.

\bibitem{Aranda-Bricaire1995}
{\sc E.~Aranda-Bricaire, C.~Moog, and J.-B. Pomet}, {\em A linear algebraic
  framework for dynamic feedback linearization}, {IEEE} Transactions on
  Automatic Control, 40 (1995), pp.~127--132,
  \url{https://doi.org/10.1109/9.362886}.

\bibitem{Banaszuk1995}
{\sc A.~Banaszuk and J.~Hauser}, {\em Feedback linearization of transverse
  dynamics for periodic orbits}, Systems {\&} Control Letters, 26 (1995),
  pp.~95--105, \url{https://doi.org/10.1016/0167-6911(94)00110-H}.

\bibitem{Brockett1983}
{\sc R.~W. Brockett, R.~S. Millman, and H.~J. Sussmann}, {\em Differential
  geometric control theory : proceedings of the conference held at Michigan
  Technological University, June 28-July 2, 1982.}, Progress in mathematics ;
  v. 27, Birkhaeuser,, Boston, 1983.

\bibitem{Lewis2005}
{\sc F.~Bullo and A.~D. Lewis}, {\em Geometric Control of Mechanical Systems},
  vol.~49 of Texts in Applied Mathematics, Springer New York, 2005.

\bibitem{DSouza2022}
{\sc R.~D'Souza}, {\em Algorithmic transverse feedback linearization}, PhD
  thesis, University of Waterloo, 2022,
  \url{http://hdl.handle.net/10012/18173}.

\bibitem{DSouza2021c}
{\sc R.~S. D{\textquotesingle}Souza, R.~Louwers, and C.~Nielsen}, {\em
  Piecewise linear path following for a unicycle using transverse feedback
  linearization}, {IEEE} Transactions on Control Systems Technology, 29 (2021),
  pp.~2575--2585, \url{https://doi.org/10.1109/tcst.2021.3049715}.

\bibitem{DSouza2021}
{\sc R.~S. D'Souza and C.~Nielsen}, {\em {An exterior differential
  characterization of single-input local transverse feedback linearization}},
  Automatica, 127 (2021), p.~109493,
  \url{https://doi.org/10.1016/j.automatica.2021.109493}.

\bibitem{Gardner1992}
{\sc R.~Gardner and W.~Shadwick}, {\em The {GS} algorithm for exact
  linearization to {B}runovsky normal form}, {IEEE} Transactions on Automatic
  Control, 37 (1992), pp.~224--230, \url{https://doi.org/10.1109/9.121623}.

\bibitem{Gardner1991}
{\sc R.~B. Gardner and W.~F. Shadwick}, {\em An algorithm for feedback
  linearization}, Differential Geometry and its Applications, 1 (1991),
  pp.~153--158, \url{https://doi.org/10.1016/0926-2245(91)90028-8}.

\bibitem{Gill2015}
{\sc R.~J. Gill, D.~Kulic, and C.~Nielsen}, {\em Spline path following for
  redundant mechanical systems}, {IEEE} Transactions on Robotics, 31 (2015),
  pp.~1378--1392, \url{https://doi.org/10.1109/tro.2015.2489502}.

\bibitem{Hermann1963}
{\sc R.~Hermann}, {\em On the accessibility problem in control theory}, in
  International Symposium on Nonlinear Differential Equations and Nonlinear
  Mechanics, Elsevier, 1963, pp.~325--332,
  \url{https://doi.org/10.1016/b978-0-12-395651-4.50035-0}.

\bibitem{Hermann1981}
{\sc R.~Hermann}, {\em The theory of equivalence of {P}faffian systems and
  input systems under feedback}, Mathematical Systems Theory, 15 (1981),
  pp.~343--356, \url{https://doi.org/10.1007/bf01786990}.

\bibitem{Hermann1988}
{\sc R.~Hermann}, {\em Invariants for feedback equivalence and {C}auchy
  characteristic multifoliations of nonlinear control systems}, Acta
  Applicandae Mathematicae, 11 (1988), pp.~123--153,
  \url{https://doi.org/10.1007/bf00047284}.

\bibitem{Hermann1989}
{\sc R.~Hermann}, {\em Nonlinear feedback control and systems of partial
  differential equations}, Acta Applicandae Mathematicae, 17 (1989),
  pp.~41--94, \url{https://doi.org/10.1007/bf00052493}.

\bibitem{Hirschorn1981}
{\sc R.~M. Hirschorn}, {\em $({A},\mathcal{B})$-invariant distributions and
  disturbance decoupling of nonlinear systems}, {SIAM} Journal on Control and
  Optimization, 19 (1981), pp.~1--19, \url{https://doi.org/10.1137/0319001}.

\bibitem{HuntSuMeyer1983}
{\sc L.~R. Hunt, R.~Su, and G.~Meyer}, {\em Design for multi-input nonlinear
  systems}, in Differential Geometric Control Theory, R.~W. Brockett, R.~S.
  Millman, and H.~J. Sussmann, eds., Birkhäuser, 1983, pp.~268--298.

\bibitem{Isidori1995}
{\sc A.~Isidori}, {\em {Nonlinear Control Systems}}, {Springer-Verlag London},
  third~ed., 1995.

\bibitem{Lee2012}
{\sc J.~M. Lee}, {\em Introduction to Smooth Manifolds}, Springer New York,
  second~ed., 2012.

\bibitem{Marino1986}
{\sc R.~Marino}, {\em On the largest feedback linearizable subsystem}, Systems
  {\&} Control Letters, 6 (1986), pp.~345--351,
  \url{https://doi.org/10.1016/0167-6911(86)90130-1}.

\bibitem{Mehra2014}
{\sc R.~Mehra, V.~Chinde, F.~Kazi, and N.~Singh}, {\em Feedback linearization
  of single-input and multi-input control system}, {IFAC} Proceedings Volumes,
  47 (2014), pp.~5479--5484,
  \url{https://doi.org/10.3182/20140824-6-za-1003.02270}.

\bibitem{Mullhaupt2006}
{\sc P.~Mullhaupt}, {\em Quotient submanifolds for static feedback
  linearization}, Systems {\&} Control Letters, 55 (2006), pp.~549--557,
  \url{https://doi.org/10.1016/j.sysconle.2005.12.002}.

\bibitem{Nguyen2016}
{\sc Q.~Nguyen and K.~Sreenath}, {\em Exponential control barrier functions for
  enforcing high relative-degree safety-critical constraints}, in 2016 American
  Control Conference ({ACC}), {IEEE}, jul 2016,
  \url{https://doi.org/10.1109/acc.2016.7524935}.

\bibitem{Nielsen2008}
{\sc C.~Nielsen and M.~Maggiore}, {\em On local transverse feedback
  linearization}, {SIAM} Journal on Control and Optimization, 47 (2008),
  pp.~2227--2250, \url{https://doi.org/10.1137/070682125}.

\bibitem{Nijmeijer2016}
{\sc H.~Nijmeijer and A.~V.~D. Schaft}, {\em Nonlinear Dynamical Control
  Systems}, Springer New York, Feb. 2016.

\bibitem{Sastry1999}
{\sc S.~Sastry}, {\em {Nonlinear Systems}}, {Springer-Verlag New York}, 1999.

\bibitem{Schoberl2014}
{\sc M.~Sch{\"o}berl and K.~Schlacher}, {\em On an implicit triangular
  decomposition of nonlinear control systems that are 1-flat{\textemdash}{A}
  constructive approach}, Automatica, 50 (2014), pp.~1649--1655,
  \url{https://doi.org/10.1016/j.automatica.2014.04.007}.

\bibitem{Tilbury1994}
{\sc D.~Tilbury and S.~Sastry}, {\em On {G}oursat normal forms, prolongations,
  and control systems}, in Proceedings of 33rd {IEEE} Conference on Decision
  and Control, {IEEE}, 1994, \url{https://doi.org/10.1109/cdc.1994.411123}.

\bibitem{Xiao2019}
{\sc W.~Xiao and C.~Belta}, {\em Control barrier functions for systems with
  high relative degree}, in 2019 {IEEE} 58th Conference on Decision and Control
  ({CDC}), {IEEE}, dec 2019,
  \url{https://doi.org/10.1109/cdc40024.2019.9029455}.

\end{thebibliography}

\end{document}